\documentclass[11pt,oneside]{amsart}
\usepackage[top=25mm, bottom=25mm, left=20mm, right=20mm]{geometry}
\usepackage{amsfonts,amssymb,amsmath,amsthm,amsxtra}
\usepackage[T1]{fontenc}
\usepackage[utf8]{inputenc}
\usepackage{bm}
\usepackage{mathtools}
\usepackage{dsfont}
\usepackage{mathrsfs}
\usepackage{enumitem}
\allowdisplaybreaks

\usepackage{graphics}
\usepackage{hyperref}

\numberwithin{equation}{section}
\newtheorem{theorem}[equation]{Theorem}

\newtheorem{lemma}[equation]{Lemma}

\theoremstyle{definition}
\newtheorem{remark}[equation]{Remark}
\newtheorem{definition}[equation]{Definition}

\newcommand{\Mod}[1]{\ (\mathrm{mod}\ #1)}

\begin{document}
\title[Roth's theorem and the Hardy--Littlewood majorant problem for thin subsets of Primes]{Roth's theorem and the Hardy--Littlewood majorant problem\\ for thin subsets of Primes}
\author{Leonidas Daskalakis}
\thanks{Department of Mathematics, Rutgers University, Leonidas Daskalakis is supported by the NSF grant DMS-2154712.}

\maketitle
\begin{abstract}
We introduce a wide class of deterministic subsets of primes \textit{of zero relative density} and we prove Roth's Theorem in these sets, namely, we show that any subset of them with positive relative upper density contains infinitely many non-trivial three-term arithmetic progressions. We also prove that the Hardy--Littlewood majorant property holds for these subsets of primes. Notably, our considerations recover the results for the Piatetski--Shapiro primes for exponents close to $1$, which are primes of the form $\lfloor n^c\rfloor$ for a fixed $c>1$.
\end{abstract}
\section{Introduction}
For any arithmetical set $A$ we call $\limsup_{N\to\infty}\frac{|A\cap [1,N]|}{N}$ its upper density. Also, we denote by $r_3(N)$ the Erd\"os--Tur\'an constant, namely, the density of the largest subset of $\{\,1,\dotsc,N\,\}$ with no non-trivial three-term arithmetic progressions. Before stating the main Theorems of the present work we provide some brief historical remarks. In 1953 Roth \cite{Roth} proved that any subset of the integers with positive upper density contains a non-trivial three-term arithmetic progression. In fact, his result is quantitative since he showed that $r_3(N)=O((\log\log N)^{-1})$. In the last 50 years, the result has been dramatically improved (\cite{HB,SZ,BG1,BG2,SA1,SA2,Bloom,BLSI}) and recently, a striking leap has been made by Bloom and Sisask \cite{Bl-Si}. They showed that $r_3(N)=O(\log ^{-1-c}N)$ for some $c>0$, breaking the logarithmic barrier and proving that any arithmetic set $A$ such that $\sum_{n\in A}\frac{1}{n}=\infty$ contains non-trivial three-term arithmetic progressions.

This corollary implies that any $A\subseteq \mathbb{P}$ with positive relative upper density, i.e. $\limsup_{N\to\infty}\frac{|A\cap [1,N]|}{|\mathbb{P}\cap [1,N]|}>0$, contains infinitely many non-trivial three-term arithmetic progressions. This result was already proven in 2003 in the seminal work of Green \cite{BenGreen} but the recent work of Bloom and Sisask establishes Roth's Theorem in the primes using only their density and not their specific structure, which was exploited in the paper of Green. In the introduction of Green's paper he remarks that ``it is possible, indeed probable, that Roth’s theorem in the primes is true on grounds of density alone,'' and the breakthrough of Bloom and Sisask affirms Green's conjecture.

The same year Roth proved his result \cite{Roth}, Piatetski--Shapiro introduced certain thin subsets of primes. For $\gamma<1$ sufficiently close to $1$, the Piatetski--Shapiro primes of type $\gamma$, are defined to be 
\[
\mathbb{P}_{\gamma}=\mathbb{P}\cap \{\,\lfloor n^{1/\gamma}\rfloor : \,n\in\mathbb{N}\,\}
\]
and he showed \cite{PS} that for $\gamma\in (11/12,1)$ we have 
\[
\mathbb{P}_{\gamma}\cap[1,x]\sim \frac{x^{\gamma}}{\log x}, \text{ as }x\to \infty
\]
Loosely speaking the statement above should be understood as an independence statement, i.e. being a prime and being of the form $\lfloor n^{1/\gamma}\rfloor$ are \textit{independent events} since the density of their intersection is the product of their densities.

Recently, Roth's Theorem was established in the Piatetski--Shapiro primes, see \cite{MMR}, for $\gamma$ close to $1$. In fact, the aforementioned paper proves Roth's Theorem in primes of the form $\lfloor n^{1/\gamma}\ell(n) \rfloor$ where $\ell$ is a certain kind of slowly varying function, for example any iterate of $\log$, see Definitions~$\ref{def1},\ref{def0}$ below. 

One of our main results is a natural extension of Roth's Theorem in the Piatetski--Shapiro primes and to state it, we need to introduce two important families of functions.

\begin{definition}\label{def1}
	Fix $x_0\ge1$ and let $\mathcal{L}$ denote the set of all functions $\ell\colon[x_0,\infty)\to[1,\infty)$ such that
	\[
	\ell(x)=\exp\bigg(\int_{x_0}^x\frac{\vartheta(t)}{t}dt\bigg)
	\] 
	where $\vartheta\in\mathcal{C}^2([x_0,\infty))$ is a real-valued function satisfying
\[
\vartheta(x)\to 0\,\,\text{, }x\vartheta'(x)\to 0\,\,\text{, }x^2\vartheta''(x)\to 0\,\,\text{ as }x\to \infty
\]
\end{definition} 

\begin{definition}\label{def0}
	Fix $x_0\ge1$ and let $\mathcal{L}_0$ denote the set of all functions $\ell\colon[x_0,+\infty)\to[1,+\infty)$ such that
	\[
	\ell(x)=\exp\bigg(\int_{x_0}^x\frac{\vartheta(t)}{t}dt\bigg)
	\] 
	where $\vartheta\in\mathcal{C}^2([x_0,+\infty))$ is a positive and decreasing function satisfying
\[
\vartheta(x)\to 0\,\,\text{, }\frac{x\vartheta'(x)}{\vartheta(x)}\to 0\,\,\text{, }\frac{x^2\vartheta''(x)}{\vartheta(x)}\to 0\,\,\text{ as }x\to \infty\text{,}
\]
and such that for all $\varepsilon>0$ we have $\vartheta(x)\gtrsim_{\varepsilon}x^{-\varepsilon}$ and $\lim_{x\to\infty}\ell(x)=\infty$.

\end{definition}

Note that $\mathcal{L}_0\subseteq \mathcal{L}$. We may think of these families as slowly varying functions and now we define a family of regularly varying functions.

\begin{definition}
Fix $x_0\ge 1$, $c\in(1,\infty)$ and let $\mathcal{R}_c$ be the set of all functions $h\colon [x_0,+\infty)\to [1,+\infty)$ such that $h$ is strictly increasing, convex and of the form $h(x)=x^c\ell(x)$ for some $\ell\in\mathcal{L}$. We define $\mathcal{R}_1$ analogously, but with the extra assumption that $\ell\in\mathcal{L}_0$. 
\end{definition} 

We are now ready to give the definitions of the arithmetic sets we are interested in. Let $c_1,c_2\in[1,2)$ and let us fix $h_1$ and $h_2$ in $\mathcal{R}_{c_1}$ and $\mathcal{R}_{c_2}$ respectively. Let $\varphi_1$ and $\varphi_2$ be the inverses of $h_1$ and $h_2$. For convenience, let $\gamma_1=1/c_1$ and $\gamma_2=1/c_2$.  Let us fix a function $\psi\colon [1,+\infty)\to(0,1/2]$, $\psi\in\mathcal{C}^2([1,+\infty))$ such that
\[
\psi(x)\sim\varphi_2'(x)\text{ , }\,\,\psi'(x)\sim\varphi_2''(x)\text{ , }\,\,\psi''(x)\sim\varphi_2'''(x)\text{ as }x\to \infty
\]
We can now define $B_+=\{\,n\in\mathbb{N}:\{\varphi_1(n)\}<\psi(n)\,\}$ and $B_-=\{\,n\in\mathbb{N}:\{-\varphi_1(n)\}<\psi(n)\,\}$, where $\{x\}=x-\lfloor x \rfloor$.

Those sets have been introduced and studied in \cite{HLMP}, where the authors proved that the Hardy–Littlewood majorant property holds for them, see Theorem 1 and 2 in \cite{HLMP}, page 4, as a Corollary of a restriction theorem. Let us denote $\mathbb{P}\cap B_+$ by $\mathbb{P}_{B_+}$ and $\mathbb{P}_{B{_-}}$ analogously. Note that these sets may be thought of as generalized Piatetski--Shapiro primes. To see this note that
\[
n\in B_-\iff \exists  m\in\mathbb{N}\colon\,0\le m-\varphi_1(n)<\psi(n)\iff \exists m\in\mathbb{N}\colon\,\varphi_1(n)\le m < \varphi_1(n)+\psi(n)\iff
\]  
\[
\exists m\in\mathbb{N}\colon\,n\le h_1(m)<h_1(\varphi_1(n)+\psi(n))\iff \exists m\in\mathbb{N}\colon\,h_1(m)\in [n,h_1(\varphi_1(n)+\psi(n)))
\]
For $\gamma\in(0,1)$, $h_1(x)=h_2(x)=x^{1/\gamma}$ and $\psi(x)=\varphi_1(x+1)-\varphi_1(x)$ the last condition becomes $m^{1/\gamma}\in [n,n+1)$ or $n=\lfloor m^{1/\gamma}\rfloor$. Thus in that case $\mathbb{P}_{\gamma}=\mathbb{P}_{B_-}$ and moreover, any set $\{\,\lfloor h(m) \rfloor\colon m\in\mathbb{N}\,\}$, $h\in\mathcal{R}_c$ can brought in the form $B_-$ by similar appropriate choices. This means that Theorem~$\ref{T2}$ below implies Roth's Theorem in sets $\mathbb{P}\cap\{\,\lfloor h(m) \rfloor\colon m\in\mathbb{N}\,\}$, $h\in\mathcal{R}_c$ for $c$ close to $1$, see \cite{MMR}, and in particular in the Piatetski--Shapiro primes.

One of the main results of our paper is the following.
\begin{theorem}[Roth's theorem in the set $\mathbb{P}_{B_{\pm}}$]\label{T2}
	Let $c_1,c_2\in[1,95/94)$. Then any $A\subseteq \mathbb{P}_{B_{\pm}}$ with positive relative upper density  
	\[
	\text{i.e. }\limsup_{N\to+\infty}\frac{|A\cap[N]|}{|\mathbb{P}_{B_{\pm}}\cap[N]|}>0
	\] 
contains infinitely many non-trivial three-term arithmetic progressions.
\end{theorem}
We note that one can also obtain with much less difficulty a Roth Theorem in $B_{\pm}$.
\begin{theorem}[Roth's theorem in the set $B_{\pm}$]\label{T1}
	Let $c_1,c_2\in[1,16/15)$. Then any $A\subseteq B_{\pm}$ with positive relative upper density, 
	\[
	\text{i.e. }\limsup_{N\to+\infty}\frac{|A\cap[N]|}{|B_{\pm}\cap [N]|}>0\text{,}
	\] 
contains infinitely many non-trivial three-term arithmetic progressions.
\end{theorem}
Before making some remarks about Theorem~$\ref{T2}$ and discussing the strategy of our proof, we would like to comment on the sophisticated nature of sets $B_+$ and $B_-$. Let us restrict our attention to the sets $B_+$ which we call $B$ from now on, since the results for the sets $B_-$ are of equal difficulty. Firstly, note that 
	\[
	n\in B\iff \exists  m\in\mathbb{N}:\,0\le \varphi_1(n)-m<\psi(n)\iff \exists m\in\mathbb{N}:\,m\in(\varphi_1(n)-\psi(n),\varphi_1(n)]
	\]  
	Now assume that $n\in B$, and $m\in \mathbb{N}$ is such that $m\in(\varphi_1(n)-\psi(n),\varphi_1(n)]$ and assume that $n_0$ is the smallest integer such that $m\in(\varphi_1(n_0)-\psi(n_0),\varphi_1(n_0)]$. Here, even in simple examples, we should expect that $B$ will have a lot of consecutive integers after $n_0$. For example, a simple application of the Mean Value Theorem shows that if we let $\varphi_1$ be an inverse of a function in $\mathcal{R}_c$, $\varphi_2=C100\varphi_1$, where $C$ is the doubling constant of $\varphi_1'$, namely, $\varphi_1'(x)\le C \varphi_1'(2x)$,  and $\psi=\varphi_2'$, then the set $B$ will contain infinitely many full blocks of 100 consecutive integers. Such a set $B$ stands in sharp contrast to the sets of the form $\{\,\lfloor h(m) \rfloor\colon m\in\mathbb{N}\,\}$, $h\in\mathcal{R}_c$, since the gaps between members of such sets tend to infinity.
	In general, the constant $\sup_{x\in[1,\infty)}\frac{\psi(x)}{\varphi_1'(x)}$ determines an important qualitative aspect of the sets $B$. Loosely speaking, for big intervals of integers where the ratio $\frac{\psi(x)}{\varphi_1'(x)}$ is bigger than $L$, we expect that $B$ will contain blocks of length at least $L/C$, where $C$ is the doubling constant of $\varphi_1'$. Even in the simpler case where $\varphi_1\simeq \varphi_2$, $B$ could contains blocks of various oscillating lengths!
	
We hope that the discussion above demonstrates how rich the family of sets $B$ is and we now proceed with some comments about Theorem~$\ref{T2}$.
\begin{remark}
We note that unlike Roth's Theorem in the primes, no improvement of the bound of $r_3(N)$ can ever imply Theorem~$\ref{T2}$ or Theorem~$\ref{T1}$ since the density of $\mathbb{P}_B$ can decay polynomially and  a result of Behrend \cite{BE} shows that there exists  an absolute constant $C>0$ such that $r_3(N)\ge e^{-C\sqrt{\log N}}$. This means that any proof of our result cannot rely solely on density considerations and must use the underlying structure of $\mathbb{P}_B$. That is the reason why Green's work is extremely useful here.
\end{remark}
\begin{remark}
Our proof of Theorem~\ref{T2} works for $c_1,c_2\ge 1$ such that $16(1-\gamma_1)+79(1-\gamma_2)<1$, where $\gamma_1=1/c_1$ and  $\gamma_2=1/c_2$, but we chose the more strict condition $c_1,c_2\in[1,95/94)$ for the sake of simplicity. One could optimize the constants of the proof and require slightly weaker assumptions but in an effort to keep the exposition reasonable we avoided stating the sharpest result derivable by our methods, since, unfortunately, even the sharpest result we can derive here is far from the one we believe to be true, namely, that the result holds for the full range $(1,2)$.
\end{remark}
The strategy of our proof of Theorem~$\ref{T2}$ is the following: we will prove a restriction theorem for $\mathbb{P}_B$, see Proposition~$\ref{RestOur}$, and then we will use a transference principle in a similar manner to \cite{BenGreen} and \cite{MMR} to conclude the proof. For the restriction theorem for the set $\mathbb{P}_B$ we use the estimates for exponential sums of Lemma~$\ref{PR2}$ together with a Tomas--Stein $T T^*$ argument to reduce the matter to the restriction theorem for primes that can be found in the work of Green \cite{BenGreen}. Vaughan’s identity will play a crucial role in the proof of Lemma~$\ref{PR2}$, and we note that a main tool for estimating exponential sums appearing in that proof will be Van der Corput's inequality. In section~$\ref{trpr}$, we use a general transference principle to bring the problem to $\mathbb{Z}_N=\mathbb{Z}/N\mathbb{Z}$ where finite Fourier Analysis together with the restriction theorem for $\mathbb{P}_B$ will be used to estimate certain trilinear forms. We conclude the proof of Theorem~$\ref{T2}$ by following an argument originally due to Varnavides \cite{VARN} to obtain a lower bound for these trilinear forms. Finally, we note that similarly to Green's result, our proof of Theorem~$\ref{T2}$ is also quantitative, although the bounds one may obtain from our methods are far from optimal (see the end of section~$\ref{trpr}$). Roth's Theorem in the sets $B$ can be used as a toy model since the strategy is identical in both cases but the situation is much simpler there since the exponential estimates that lead to the restriction Theorem for those sets are immediate corollaries of results from \cite{HLMP}. We only give a brief sketch for the proof of Theorem~$\ref{T1}$. 
\\

The second main result of the present work is proving that the sets $\mathbb{P}_B$ obey the so-called Hardy--Littlewood majorant property, namely, that the following Theorem holds.

\begin{theorem}[Hardy--Littlewood majorant property for $\mathbb{P}_B$]\label{HLMPPB}
Let $c_1\in[1,16/15)$, $c_2\in[1,17/16)$ and $r>2+\frac{62-62\gamma_2}{16\gamma_1+17\gamma_2-32}$. There exists a positive constant $C=C(r,h_1,h_2,\psi)$ such that for any $N\in\mathbb{N}$ and any $(a_n)_{n\in\mathbb{N}}$ sequence of complex numbers such that $|a_n|\le 1$ for all $n\in\mathbb{N}$, we have
\[
\bigg\Vert\sum_{p\in \mathbb{P}_B\cap[N]} a_pe^{2 \pi i p \xi}\bigg\Vert_{L^r(\mathbb{T})}\le C\bigg\Vert\sum_{p\in \mathbb{P}_B\cap[N]} e^{2 \pi i p \xi} \bigg\Vert_{L^r(\mathbb{T})} 
\] 
\end{theorem}
Some brief historical remarks are in order. It was conjectured by Hardy and Littlewood \cite{HL} that for any $p\ge 2$, there exists a positive constant $C_p$ such that for any sequence of complex numbers $(a_n)_{n\in\mathbb{N}}$ bounded by $1$ and any finite set $A\subseteq\mathbb{N}$, we get 
\begin{equation}\label{HLMPI}
\bigg\Vert\sum_{n\in A} a_ne^{2 \pi i n \xi}\bigg\Vert_{L^p(\mathbb{T})}\le C_p\bigg\Vert\sum_{n\in A} e^{2 \pi i p \xi} \bigg\Vert_{L^p(\mathbb{T})}
\end{equation}
Parseval's identity shows that one may take $C_{2k}=1$ for any $k\in\mathbb{N}$, nevertheless, this conjecture fails for any $p>2$ which is not an even integer, see \cite{BACH}. While the full conjecture may not be true, there has been an effort to quantify that failure (see \cite{M1} for precise formulations and connections to the restriction conjecture for the Fourier transform on $\mathbb{R}^d$, see \cite{HLMP} for a brief exposition on the matter, and see \cite{multidiman} for multi-dimensional results). Simultaneously, some efforts have been made to find specific infinite arithmetic sets $A$ where either inequality $\ref{HLMPI}$ does hold for any $A\cap [N]$ with $C_p$ independent of $N$ or it fails, but nevertheless we have sufficiently good estimates for the growth of $C_p(N)$, see \cite{HLMPR,BenGreen,MMR,HLMP}. Here, sufficiently good estimates means acceptable in the context of the connections of the Hardy--Littlewood majorant problem and the restriction conjecture, see \cite{M1,HLMPR}. It is worth mentioning that in contrast to the seminal work of \cite{HLMPR}, where the behavior of $C_p(N)$ was studied for random sets, we concern ourselves with the Hardy--Littlewood majorant problem for a wide class of \text{\it{deterministic}} sets, similarly to \cite{BenGreen,MMR,HLMP}.

Finally, we note that variants of this property can play an important role in some combinatorial arguments. For example, establishing Roth's Theorem in the primes (as well as the Piatetski--Shapiro primes) involved proving a suitable discrete variant of the majorant property.
\\

Both Theorem~$\ref{T2}$ and Theorem~$\ref{HLMPPB}$ rely heavily on the restriction Theorem for the sets $\mathbb{P}_B$, see Theorem~$\ref{RestOur}$. The most technical part of establishing this restriction Theorem is the estimate $\ref{EXT2}$ of the following Lemma which we prove in the last section. Lemma~\ref{PR2} combined with Bourgain's restriction Theorem for the primes, see \cite{BenGreen}, page 3, will lead to the desired restriction Theorem for the sets $\mathbb{P}_B$.

\begin{lemma}\label{PR2} Let $c_1\in [1,16/15)$ and $c_2\in [1,17/16) $ and let $\gamma_1=1/c_1$ and $\gamma_2=1/c_2$. Let $a,q\in \mathbb{Z}$ such that $0\le a\le q-1$ and $(a,q)=1$. Then for every $\chi>0$ such that $16(1-\gamma_1)+17(1-\gamma_2)+31\chi\le1$ there exists $\chi'>0$ such that:

 	\begin{equation}\label{EXT2}
 		\sum_{\substack{p\in \mathbb{P}_B\cap [N] \\ p\equiv a\Mod q}} \psi(p)^{-1}\log(p) e^{2\pi i p\xi}= \sum_{\substack{p\in \mathbb{P}\cap[N] \\ p\equiv a \Mod q}} \log(p)e^{2\pi i p\xi} +O(N^{1-\chi-\chi'})
 	\end{equation}
where the implied constant does not depend on $\xi,N,a,q$.
 	
\end{lemma}
The curious reader might find of independent interest the following intermediate result which implies the counterpart of the Prime Number Theorem for the set $\mathbb{P}_B$, which could be interpreted as an independence statement in the same way the asymptotic formula of the Piatetski--Shapiro primes was understood.

\begin{theorem}\label{SFPB}
	Let $c_1\in[1,16/15)$ and $c_2\in[1,17/16)$, such that $16(1-\gamma_1)+17(1-\gamma_2)<1$. Let $D>0$, $ N,m,b\in\mathbb{N}$ such that $b\le m-1$, $(b,m)=1$ and $m\le \log^D(N)$, then
	\begin{equation}\label{Firstmine}
	\psi_B(N;m,b)=\sum_{\substack{p\in\mathbb{P}\cap B\cap [N]\\p\equiv b\Mod{m}}}\log(p)=\frac{\int_1^N \psi(t)dt}{\phi(m)}+O_D\bigg(\frac{\varphi_2(N)}{\log^D(N)}\bigg)
	\end{equation}
and if $m\le \log(N)$ then
	\begin{equation}\label{Thirdproved}
	\pi_B(N;m,b)=\sum_{\substack{p\in\mathbb{P}\cap B\cap [N]\\p\equiv b\Mod{m}}} 1=\frac{\int_1^N \psi(t)dt}{\phi(m)\log(N)}+O\bigg(\frac{\varphi_2(N)}{\log^2(N)}\bigg)
	\end{equation}
	where the implied constant depends only on $D$, $h_1,h_2$, and $\psi$.
\end{theorem}
\subsection{Notation}
We note that 3AP stands for non-trivial three-term arithmetic progression and any set with no 3APs will be called 3AP-free. We denote by $C$ a positive constant that may change from occurrence to occurrence. If $A,B$ are two non-negative quantities, we write $A\lesssim B$ or $B \gtrsim A$  to denote that there exists a positive constant $C$ such that $A\le C B$. Whenever $A\lesssim B$ and $A\gtrsim B$ we write $A\simeq B$. For two complex-valued functions $f,g$ we write $f\sim g$ to denote that $\lim_{x\to\infty}\frac{f(x)}{g(x)}=1$. For any topological space $X$, we denote by $\mathcal{C}(X)$ the set of all complex-valued continuous functions, and for any finitely supported $f\colon \mathbb{Z}\to \mathbb{C}$ we define the Fourier Transform
\[
\mathcal{F}[f](\xi)=\sum_{k\in\mathbb{Z}}f(k)e^{2\pi i k \xi}\text{, for all }\xi\in\mathbb{T}
\]
For any $g\colon \mathbb{Z}_N=\mathbb{Z}/N\mathbb{Z}\to \mathbb{C}$ we define the finite Fourier Transform and the inverse Fourier Transform
\[
\mathcal{F}_{\mathbb{Z}_N}[g](\xi)=\sum_{k\in\mathbb{Z}_N}g(k)e^{\frac{-2\pi i k \xi}{N}}\text{ and }\mathcal{F}^{-1}_{\mathbb{Z}_N}[g](\xi)=\sum_{k\in\mathbb{Z}_N}g(k)e^{\frac{2\pi i k \xi}{N}}\text{, for all }\xi\in\mathbb{Z}_N
\]and note that the following Fourier Inversion formula holds $\mathcal{F}^{-1}_{\mathbb{Z}_N}\big[\mathcal{F}_{\mathbb{Z}_N}[g]\big](\xi)=Ng(\xi)$.

\section{Restriction Theorem for the set $\mathbb{P}_B$}\label{RestforPbeta}This section is devoted to proving the restriction Theorem for the sets $\mathbb{P}_B$, see Definition~$\ref{DefOur}$ and Theorem~$\ref{RestOur}$. The restriction theorem for the primes together with the exponential estimates of Lemma~$\ref{PR2}$ will be the key elements of our proof. Here we fix two constants $c_1,c_2$ such that $c_1\in[1,32/31)$ and $c_2\in[1,34/33)$, as well as $h_1$, $h_2\in\mathcal{R}_{c_1}$ and $\mathcal{R}_{c_2}$ respectively and $\psi$ as in the introduction and all the implied constants may depend on them. The reader should compare this work with the Section 4 of \cite{MMR}, Section 2 of \cite{BenGreen} and Section 3 of \cite{HLMP}.

Before stating Bourgain's Restriction Theorem for the primes, which will be essential for our argument, we introduce the following notation from \cite{BenGreen}.

\begin{definition}
	For any $N\in\mathbb{N}$ and $m,b\in\mathbb{N}$ such that $0\le b \le m-1$ and $(m,b)=1$ and $m\le \log(N)$, let \[\Lambda_{b,m,N}=\{\,n\in\{\,1,\dotsc,N\,\}:\,mn+b\in\mathbb{P}\,\}
	\]
	and
	\[
	\lambda_{b,m,N}(n)=\left\{
	\begin{array}{ll}
		\frac{\phi(m)\log(mn+b)}{mN}, n\in \Lambda_{b,m,N} \\
		0, n\notin \Lambda_{b,m,N}
	\end{array} 
	\right.
	\]
	where $\phi$ denotes the Euler's totient function. Also, let's define a function $T_{b,m,N}\colon \mathcal{C}(\Lambda_{b,m,N})\to \mathcal{C}(\mathbb{T})$ such that
	\[
	T_{b,m,N}(f)(\xi)=\mathcal{F}[f\lambda_{b,m,N}](\xi)=\sum_{k\in \mathbb{Z}}f(k)\lambda_{b,m,N}(k)e^{2\pi i k \xi}\text{ for all }\xi\in\mathbb{T}
	\]
	
\end{definition}
We will abuse notation and sometimes treat $\lambda_{b,m,N}$ as a measure on $\Lambda_{b,m,N}$ in the obvious way, namely
\[
\lambda_{b,m,N}(A)=\sum_{n\in A}\lambda_{b,m,N}(n)\text{ for all }A\subseteq \Lambda_{b,m,N}
\]
The Siegel–Walfisz Theorem allows us to think that $\lambda_{b,m,N}$ is, loosely speaking, a probability measure on $\Lambda_{b,m,N}$. More precisely we have the following theorem.

\begin{theorem}[Siegel–Walfisz]
	Let $D>0$, $N,m,b\in\mathbb{N}$ such that $b \le m-1$, $(m,b)=1$ and $m\le \log^D(N)$, then
	\begin{equation}\label{S-W1}
		\psi_1(N;m,b)=\sum_{\substack{n\in[N]\\n\equiv b\Mod{m}}}\Lambda(n)=\frac{N}{\phi(m)}+O_D\bigg(\frac{N}{\log^D(N)}\bigg)
	\end{equation}
	and 
	\begin{equation}\label{S-W2}
		\psi_2(N;m,b)=\sum_{\substack{p\in\mathbb{P}_N\\p\equiv b\Mod{m}}}\log(p)=\frac{N}{\phi(m)}+O_D\bigg(\frac{N}{\log^D(N)}\bigg)
	\end{equation}
	where the implied constant depends only on $D$.
\end{theorem}
\begin{proof}
	For the proof of $\ref{S-W1}$, see \cite{IWKO}, Corollary 5.29, page 124. Through standard elementary estimates we obtain
	\[
	\big|\psi_1(N,m,b)-\psi_2(N,m,b)\big|\lesssim \sqrt{N}
	\]
which implies the desired result.	
\end{proof}
Applying the result for $D=1$ shows that $\lambda_{b,m,N}(\Lambda_{b,m,N})=\sum_{n\in\Lambda_{b,m,N}}\lambda_{b,m,N}(n)\to 1$, as $N\to \infty$, justifying the previous heuristic.
\begin{theorem}[Bourgain--Green]\label{B-G}Suppose that $r>2$ is a real number. Then there exists a positive constant $C_r$ such that for all functions $f\colon\Lambda_{b,m,N}\to\mathbb{C}$ we have
	\[
	||T_{b,m,N}(f)||_{L^r(\mathbb{T})}\le C_rN^{-1/r}||f||_{L^2(\Lambda_{b,m,N},\lambda_{b,m,N})}	
	\]
\end{theorem}
\begin{proof}
	This result can be found in \cite{BenGreen}, see Theorem 2.1, page 3.
\end{proof}

We now introduce the sets and measures analogous to $\Lambda_{b,m,N}$ and $\lambda_{b,m,N}$ that will allow us to state the restriction Theorem for the set $\mathbb{P}_B$.
\begin{definition}\label{DefOur}
	For any $N\in\mathbb{N}$ and $m,b\in\mathbb{N}$ such that $0\le b \le m-1$ and $(m,b)=1$ and $m\le \log(N)$, let \[\mathcal{P}_{b,m,N}=\{\,n\in\{\,1,\cdots,N\,\}:\,mn+b\in\mathbb{P}_B\,\}
	\]
	and
	\[
	\rho_{b,m,N}(n)=\left\{
	\begin{array}{ll}
		\psi(mn+b)^{-1}\cdot\frac{\phi(m)\log(mn+b)}{mN}, n\in \mathcal{P}_{b,m,N} \\
		0, n\notin \mathcal{P}_{b,m,N}
	\end{array} 
	\right.
	\]
	where $\phi$ denotes the Euler's totient function. Also, let's define a function $T^B_{b,m,N}\colon \mathcal{C}(\mathcal{P}_{b,m,N})\to \mathcal{C}(\mathbb{T})$ such that
	\[
	T^B_{b,m,N}(f)(\xi)=\mathcal{F}[f\rho_{b,m,N}](\xi)=\sum_{k\in \mathbb{Z}}f(k)\rho_{b,m,N}(k)e^{2\pi i k \xi}\text{ for all }\xi\in\mathbb{T}
	\]
	
\end{definition}

\begin{theorem}[Restriction Theorem for $\mathbb{P}_B$]\label{RestOur}
	Let $N\in\mathbb{N}$ and $m,b\in\mathbb{Z}$ such that $0\le b \le m-1$, $(m,b)=1$ and $m\le \log(N)$. Then for any real number $r>2+\frac{62-62\gamma_2}{16\gamma_1+17\gamma_2-32}$ there exists a positive constant $C=C(r,h_1,h_2,\psi)$ such that for all $f\colon \mathcal{P}_{b,m,N} \to \mathbb{C}$ we have
	\[
	||T^B_{b,m,N}(f)||_{L^r(\mathbb{T})}\le CN^{-1/r}||f||_{L^2(\mathcal{P}_{b,m,N},\rho_{b,m,N})}
	\]
\end{theorem}
\begin{proof}
	We will use the $TT^*$ argument and interpolation. Firstly, note that for $g\in L^1(\mathbb{T})$ and $f\in\mathcal{C}(\mathcal{P}_{b,m,N})$ we have
	\[
	\langle T^B_{b,m,N}(f),g\rangle_{L^2(\mathbb{T})}=\int_0^1 \mathcal{F}[f\rho_{b,m,N}](\xi)\overline{g(\xi)}d\xi=\sum_{n\in\mathbb{Z}}f(n)\rho_{b,m,N}(n)\overline{\hat{g}(n)}
	\]
	\[
	=\sum_{n\in \mathcal{P}_{b,m,N}}f(n)\overline{\hat{g}(n)1_{\mathcal{P}_{b,m,N}}(n)} \rho_{b,m,N}(n)=\langle f,\hat{g}1_{\mathcal{P}_{b,m,N}}\rangle_{L^2(\mathcal{P}_{b,m,N},\rho_{b,m,N})}
	\]
	We remark that $(\mathcal{C}(\mathcal{P}_{b,m,N}))^*\cong \mathcal{C}(\mathcal{P}_{b,m,N})$ as Banach Spaces through the map $h\to \Phi_h$, where $\Phi_h(f)=\langle f,h\rangle_{L^2(\mathcal{P}_{b,m,N},\rho_{b,m,N})}$ and also that any $L^p(\mathbb{T})$, $p\ge 1$ can be embedded into $(C(\mathbb{T}))^*$  via the map $k\to \Psi_k$, 
	where $\Psi_k(g)=\langle g,k \rangle_{L^2(\mathbb{T})}$. We have shown that $\Psi_g(T^B_{b,m,N}(f))=\Phi_{\hat{g}1_{\mathcal{P}_{b,m,N}}}(f)$ and thus $(T^B_{b,m,N})^*(\Psi_g)=\Phi_{\hat{g}1_{\Lambda_{b,m,N}}}$ and we will abuse notation and write $(T^B_{b,m,N})^*(g)(n)=\hat{g}(n)1_{\mathcal{P}_{b,m,N}}(n)$. Let's notice that
	\[
	T^B_{b,m,N}(T^B_{b,m,N})^*(g)=T^B_{b,m,N}(\hat{g}1_{P_{b,m,N}})=\mathcal{F}(\hat{g}1_{\mathcal{P}_{b,m,N}}\rho_{b,m,N})=g*\mathcal{F}[\rho_{b,m,N}]
	\]
	and we note that a similar calculation shows that $T_{b,m,N}T^*_{b,m,N}(g)=g*\mathcal{F}[\lambda_{b,m,N}]$. Similarly to the previous restriction Theorem, it is enough to show that \begin{equation}\label{GoalN}
		||T^B_{b,m,N}(T^B_{b,m,N})^*||_{L^{r'}(\mathbb{T})\to L^r(\mathbb{T})}\le CN^{-2/r}
	\end{equation}To see this, let $f\in L^r(\mathcal{P}_{b,m,N},\rho_{b,m,N})$ and $g\in L^{r'}(\mathbb{T})$, then
	\[
	|\langle T^B_{b,m,N}(f),g\rangle_{L^2(\mathbb{T})}|=|\langle f,(T^B_{b,m,N})^*(g)\rangle_{L^2(\mathcal{P}_{b,m,N},\rho_{b,m,N})}|\le ||(T^B_{b,m,N})^*(g)||_{L^2(\mathcal{P}_{b,m,N},\rho_{b,m,N})}||f||_{L^2(\mathcal{P}_{b,m,N},\rho_{b,m,N})}
	\]
	and also
	\[
	||(T^B_{b,m,N})^*(g)||_{L^2(\mathcal{P}_{b,m,N},\rho_{b,m,N})}^2=\langle T^B_{b,m,N}(T^B_{b,m,N})^*(g),g  \rangle_{L^2(\mathbb{T})}\le
	\]
	
	\[ ||T^B_{b,m,N}(T^B_{b,m,N})^*(g)||_{L^r(\mathbb{T})}||g||_{L^{r'}(\mathbb{T})}\le ||T^B_{b,m,N}(T^B_{b,m,N})^*||_{L^{r'}(\mathbb{T})\to L^r(\mathbb{T}) }||g||_{L^{r'}(\mathbb{T})}^2  
	\]
	Thus
	\[
	|\langle T^B_{b,m,N}(f),g\rangle_{L^2(\mathbb{T})}|\le ||T^B_{b,m,N}(T^B_{b,m,N})^*||^{1/2}_{L^{r'}(\mathbb{T})\to L^r(\mathbb{T}) }||g||_{L^{r'}(\mathbb{T})}  ||f||_{L^2(\mathcal{P}_{b,m,N},\rho_{b,m,N})}
	\]which justifies the fact that proving $\ref{GoalN}$ suffices for concluding our proof. We note that
	\[
	||T^B_{b,m,N}(T^B_{b,m,N})^*(g)||_{L^{r}(\mathbb{T})}=||g*\mathcal{F}[\rho_{b,m,N}]||_{L^{r}(\mathbb{T})}\le ||g*\mathcal{F}[\lambda_{b,m,N}]||_{L^{r}(\mathbb{T})}+||g*\mathcal{F}[\rho_{b,m,N}-\lambda_{b,m,N}]||_{L^{r}(\mathbb{T})}\le 
	\]
	\[
	||T_{b,m,N}T^*_{b,m,N}(g)||_{L^{r}(\mathbb{T})}+||g*\mathcal{F}[\rho_{b,m,N}-\lambda_{b,m,N}]||_{L^{r}(\mathbb{T})}\le
	\]
	\begin{equation}\label{onlyerrorterm} ||T_{b,m,N}T^*_{b,m,N}||_{L^{r'}(\mathbb{T})\to L^{r}(\mathbb{T})}||g||_{L^{r'}(\mathbb{T})}+||g*\mathcal{F}[\rho_{b,m,N}-\lambda_{b,m,N}]||_{L^{r}(\mathbb{T})}
	\end{equation}
	By the proof of Bourgain–Green’s theorem, see (2,7) in page 4 in \cite{BenGreen}, we know that there exists a positive constant $C'_r$ such that $||T_{b,m,N}T^*_{b,m,N}||_{L^{r'}(\mathbb{T})\to L^{r}(\mathbb{T})}\le C'_rN^{-2/r}$ and thus it suffices to estimate the second term in $\ref{onlyerrorterm}$ which we may think of as an error term. We show that there exists a constant $C=C(r,h_1,h_2,\psi)$ such that 
	\begin{equation}\label{finale}
		||g*\mathcal{F}[\rho_{b,m,N}-\lambda_{b,m,N}]||_{L^{r}(\mathbb{T})}\le C N^{-2/r}||g||_{L^{r'}(\mathbb{T})}
	\end{equation}
	We prove this for $r=2$, $r=\infty$ and interpolate. For $g\in L^2(\mathbb{T})$, we have
	\[
	||g*\mathcal{F}[\rho_{b,m,N}-\lambda_{b,m,N}]||_{L^2(\mathbb{T})}\le ||\hat{g}(\rho_{b,m,N}-\lambda_{b,m,N})||_{\ell^2(\mathbb{Z})}\le ||\rho_{b,m,N}-\lambda_{b,m,N}||_{\ell^{\infty}(\mathbb{Z})}||\hat{g}||_{\ell^2(\mathbb{Z})}\le
	\]
	\[
	(||\rho_{b,m,N}||_{\ell^{\infty}(\mathbb{Z})}+||\lambda_{b,m,N}||_{\ell^{\infty}(\mathbb{Z})})||g||_{L^2(\mathbb{T})}
	\]
	and $||\lambda_{b,m,N}||_{\ell^{\infty}(\mathbb{Z})}\lesssim \frac{\log(N)}{N}$, since $\lambda_{b,m,N}(n)\le \frac{\log(\log(N)N+\log(N))}{N}\lesssim \frac{\log(N)}{N}$. Now let's estimate $||\rho_{b,m,N}||_{\ell^{\infty}(\mathbb{Z})}$; for all $n\in \mathcal{P}_{b,m,N}$ we have
	\[
	|\rho_{b,m,N}(n)|=\frac{\phi(m)\log(mn+b)}{mN\psi(mn+b)}\lesssim\frac{\log(N)}{N\varphi_2'(mn+b)}\le \frac{\log(N)}{N\varphi_2'(m(N+1))}\lesssim
	\]
	\[
	\lesssim \frac{m\log(N)}{mN\varphi_2'(mN)}\lesssim   \frac{\log^2(N)}{\varphi_2(mN)}\le\frac{\log^2(N)}{\varphi_2(\log(N)N)} 
	\]
	where we have used: $\psi(x)\sim\varphi_2'(x)$, $\varphi_2$ is concave and $b\le m-1\le \log(N)$, Lemma 2.14 from \cite{MMR} for the form of $x\varphi_2'(x)$ and $\varphi_2'(x)\simeq \varphi_2(2x)$. By Lemma 2.6 in \cite{MMR}, we have that $\lim_{x\to\infty}\varphi_2(x)/x=0$ which gives $\frac{\log(N)}{N}\lesssim\frac{\log^2(N)}{\varphi_2(\log(N)N)}$	and thus
	\begin{equation}\label{est21}
		||g*\mathcal{F}[\rho_{b,m,N}-\lambda_{b,m,N}]||_{L^2(\mathbb{T})}\lesssim \frac{\log^2(N)}{\varphi_2(\log(N)N)}||g||_{L^2(\mathbb{T})}
	\end{equation}
	On the other hand
	\[
	||g*\mathcal{F}[\rho_{b,m,N}-\lambda_{b,m,N}]||_{L^{\infty}(\mathbb{T})}\le ||\mathcal{F}[\rho_{b,m,N}-\lambda_{b,m,N}]||_{L^{\infty}(\mathbb{T})}||g||_{L^1(\mathbb{T})} 
	\]
	Since $c_1\in[1,32/31)$ and $c_2\in[1,34/33)$, we have that $16(1-\gamma_1)+17(1-\gamma_2)<1$, let $\chi>0$ be such that $16(1-\gamma_1)+17(1-\gamma_2)+31\chi=1$, according to Lemma~$\ref{PR2}$, there exists $\chi'>0$ such that \ref{EXT2} is valid. Again let's estimate; let $\xi\in \mathbb{T}$ and let $\xi'=\xi/m$, we get
	
	\[
	|\mathcal{F}[\rho_{b,m,N}-\lambda_{b,m,N}](\xi)|=\Big|\sum_{n\in \mathcal{P}_{b,m,N}}\rho_{b,m,N}(n)e(n\xi) -\sum_{n\in\Lambda_{b,m,N}}\lambda_{b,m,N}(n)e(n\xi)\Big|=\]
	
	\[\frac{\phi(m)}{mN}\Big|\sum_{n\in \mathcal{P}_{b,m,N}}\psi(mn+b)^{-1}\log(mn+b)e(nm\xi')-\sum_{n\in\Lambda_{b,m,N}}\log(mn+b)e(nm\xi') \Big|\le
	\]
	\[
	\frac{\phi(m)}{mN} \Big|\sum_{\substack{k\in \mathbb{P}_B\cap[mN+b] \\ k\equiv b\Mod m}}\psi(k)^{-1}e(k\xi')-\sum_{\substack{k\in \mathbb{P}\cap[mN+b]\\k\equiv b\Mod{m}}}e(k\xi') \Big|\lesssim \frac{\phi(m)}{mN}(mN+b)^{1-\chi-\chi'}\lesssim
	\]
	\[
	 m^{1-\chi-\chi'}N^{-\chi-\chi'}\lesssim N^{-\chi-\chi'/2}
	\]
Let $\varepsilon=\chi'/2$, and note that we have
	
	\begin{equation}\label{est22}
		||g*\mathcal{F}[\rho_{b,m,N}-\lambda_{b,m,N}]||_{L^{\infty}(\mathbb{T})}\lesssim \frac{1}{N^{\chi+\varepsilon}}||g||_{L^1(\mathbb{T})}
	\end{equation}
	By applying Riesz-Thorin Interpolation to $\ref{est21}$ and $\ref{est22}$ we conclude that for any $r\in(1,+\infty)$ and $g\in L^{r'}(\mathbb{T})$ we have
	\[
	||g*\mathcal{F}[\rho_{b,m,N}-\lambda_{b,m,N}]||_{L^{r}(\mathbb{T})}\lesssim_{r}(N^{-\chi-\varepsilon})^{1-2/r}\bigg(\frac{\log^2(N)}{\varphi_2(\log(N)N)}\bigg)^{2/r}||g||_{L^{r'}(\mathbb{T})}
	\]
	Let $\varepsilon_r$ be a sufficiently small positive real number which will be chosen later. We know that $\varphi_2(x)\gtrsim_{\varepsilon_{r}} x^{\gamma_2-\varepsilon_r}$. Thus we have
	\[
	||g*\mathcal{F}[\rho_{b,m,N}-\lambda_{b,m,N}]||_{L^{r}(\mathbb{T})}\lesssim_{r,\varepsilon_{r}}N^{-
		\chi-\varepsilon+2\chi/r+2\varepsilon/r}\log^{4/r}(N)\big((\log(N)N)^{-\gamma_2+\varepsilon_r}\big)^{2/r}||g||_{L^{r'}(\mathbb{T})}=
	\]  
	\[
	\big(\log(N)\big)^{4/r-2\gamma_2/r+2\varepsilon_r/r}N^{-\chi-\varepsilon+2\chi/r+2\varepsilon/r-2\gamma_2/r+2\varepsilon_r/r}||g||_{L^{r'}(\mathbb{T})}
	\]
	We wish to have
	\begin{equation}\label{G222}
		-\chi-\varepsilon+2\chi/r+2\varepsilon/r-2\gamma_2/r+2\varepsilon_r/r<-2/r\text{, or equivalently } r>\frac{2(1+\chi+\varepsilon-\gamma_2+\varepsilon_r)}{\chi+\varepsilon}=2+\frac{2(1-\gamma_2+\varepsilon_r)}{\chi+\varepsilon}
	\end{equation}
For $\ref{G222}$ to hold, it suffices to have 
	\begin{equation}\label{Glast2}
		r>2+\frac{2(1-\gamma_2+\varepsilon_r)}{\chi}=2+\frac{62(1-\gamma_2+\varepsilon_r)}{1-16(1-\gamma_1)-17(1-\gamma_2)}=2+\frac{62-62\gamma_2+62\varepsilon_r}{16\gamma_1+17\gamma_2-32}
	\end{equation}
We have that $r>2+\frac{62-62/c_2}{16/c_1+17/c_2-32}$, and thus such a choice for $\varepsilon_r>0$ is possible and therefore we do have that $\ref{G222}$ is true, which in turn implies that
	\[
	||g*\mathcal{F}[\rho_{b,m,N}-\lambda_{b,m,N}]||_{L^{r}(\mathbb{T})}\lesssim_{r}N^{-2/r}||g||_{L^{r'}(\mathbb{T})}
	\]
	which shows $\ref{finale}$ and concludes our proof.
\end{proof}
We obtain the Hardy--Littlewood majorant property for the sets $\mathbb{P}_B$ as a Corollary. To do so, we will need some estimates for $|\mathbb{P}_B\cap [N]|$, so let's firstly prove  Theorem~$\ref{SFPB}$.
\begin{proof}[Proof of Theorem~$\ref{SFPB}$]
One may use Lemma~$\ref{PR2}$ for $\xi=0$, summation by parts, the Siegel--Walfisz Theorem and the basic properties of functions in $\mathcal{R}_c$ in order to obtain these estimates. We provide some details here. Summation by parts gives
\[
	\sum_{\substack{p\in\mathbb{P}_B\cap [N]\\p\equiv b\Mod{m}}}\log(p)=
	\psi(N)\sum_{\substack{p\in\mathbb{P}_B\cap [N]\\p\equiv b\Mod{m}}}\psi(p)^{-1}\log(p)-\int_2^N \psi'(t)\sum_{\substack{p\in\mathbb{P}_B\cap [\lfloor t\rfloor]\\p\equiv b\Mod{m}}}\psi(p)^{-1}\log(p)dt
	\]
	There exists a real number $\chi>0$ such that $16(1-\gamma_1)+17(1-\gamma_2)+31\chi<1$ and according to Proposition $\ref{PR2}$, there exists a real number $\chi'>0$ such that 
	\[
	\sum_{\substack{p\in \mathbb{P}_B\cap[L] \\ p\equiv a\Mod q}} \psi(p)^{-1}\log(p) = \sum_{\substack{p\in \mathbb{P}\cap[L] \\ p\equiv a \Mod q}} \log(p) +O(L^{1-\chi-\chi'})=\frac{L}{\phi(m)}+O_D\bigg(\frac{L}{\log^D(L)}\bigg)
	\]
	where we took into account $\ref{S-W2}$. Thus we get
	\[
	\psi_B(N;m,b)=\frac{N\psi(N)}{\phi(m)}+O_D\bigg(\frac{N\psi(N)}{\log^D(N)}\bigg)-\int_2^N \psi'(t)\bigg(\frac{\lfloor t\rfloor}{\phi(m)}+O_D\bigg(\frac{\lfloor t\rfloor}{\log^D(\lfloor t\rfloor)}\bigg)\bigg)dt
	\]and we note that by the basic properties of $\varphi_2$ and $\psi$, see Lemma~2.14 in \cite{MMR}, we get
\[
	N\psi(N)-\int_1^N\psi'(t)\lfloor t \rfloor dt=\sum_{n=1}^N\psi(n)=\int_{1}^N\psi(t)dt+O(1)\text{ and }N\psi(N)\lesssim N\varphi_2'(N)\lesssim \varphi_2(N)
	\]
Finally, we have
	\[
	\int_2^N \bigg|\psi'(t)\bigg(\frac{\lfloor t\rfloor}{\log^D(\lfloor t\rfloor)}\bigg)\bigg|dt\lesssim \int_2^N\bigg| \varphi_2''(t)\bigg(\frac{t}{\log^D(t)}\bigg)\bigg|dt\lesssim \int_2^N \varphi_2'(t)\bigg(\frac{1}{\log^D(t)}\bigg)dt=
	\]
	\[
	\varphi_2(N)/\log^D(N)-\varphi_2(2)/\log^D(2)-\int_2^N\varphi_2(t)(-D)\log^{-D-1}(t)\frac{1}{t}dt\le
	\]
	\[
	\varphi_2(N)/\log^D(N)+D\varphi_2(N)\int_{\log(2)}^{\log(N)}u^{-D-1}du\lesssim_D \varphi_2(N)/\log^D(N) 
	\]
Therefore
	\[
	\psi_B(N;m,b)=\frac{1}{\phi(m)}\int_1^N \psi(t)dt+O_D\bigg(\frac{\varphi_2(N)}{\log^D(N)}\bigg)
	\]
	and we have proved $\ref{Firstmine}$.
	For the second estimate we have
	\[
	\pi_B(N;m,b)=\frac{1}{\log(N)}\psi_B(N;m,b)-\int_2^N\psi_B(t;m,b)\bigg(\frac{1}{\log(t)}\bigg)'dt
	\]
	We have that $\psi_B(t;m,b)\lesssim \varphi_2(t)$. Let us fix positive real numbers $\varepsilon\in(0,1)$ and $\gamma_2'$ such that $\varepsilon \gamma_2<\varepsilon \gamma_2'<\gamma_2$. Then $\varphi_2(N^{\varepsilon})\lesssim N^{\varepsilon\gamma_2'}\le \varphi_2(N)/\log^2(N) $, see Lemma 2.6 in \cite{MMR}, and thus
	\[
	0<-\int_2^N\psi_B(t;m,b)\bigg(\frac{1}{\log(t)}\bigg)'dt=\int_2^N\psi_B(t;m,b)\bigg(\frac{1}{t\log^2(t)}\bigg)dt\lesssim \int_2^N \varphi_2(t)\bigg(\frac{1}{t\log^2(t)}\bigg)dt=   
	\]
	\[
	\int_{2}^{N^{\varepsilon}}\varphi_2(t)\bigg(\frac{1}{t\log^2(t)}\bigg)dt+\int_{N^{\varepsilon}}^N\varphi_2(t)\bigg(\frac{1}{t\log^2(t)}\bigg)dt\lesssim \varphi_2(N^{\varepsilon})+\frac{1}{\varepsilon\log^2(N)}\int_{N^{\varepsilon}}^N\varphi_2'(t)dt\lesssim_{\varepsilon}\frac{\varphi_2(N)}{\log^2(N)}
	\]
Finally, we combine with the previous asymptotic for $D=1$ to obtain the asymptotic $\ref{Thirdproved}$.	
\end{proof} 
We remark that $\varphi_2'\sim \psi$ implies that $\int_1^N \psi(t)dt\simeq \varphi_2(N)$ and thus $|\mathbb{P}_B\cap[N]|=\pi_{B}(N;1,0)\simeq \frac{\varphi_2(N)}{\log(N)}$ which will be a rough estimate sufficient for our purposes.
\begin{proof}[Proof of Theorem~$\ref{HLMPPB}$]

Let $N\in\mathbb{N}$ and $(a_n)_{n\in\mathbb{N}}$ be a sequence of complex numbers with $|a_n|\le 1$. Apply the restriction Theorem for $m=1$, $b=0$ and $f(n)=\frac{a_n\psi(n)}{
\log(n)}1_{\{n>1\}}(n)$. There exists  constant $C=C(r,h_1,h_2,\psi)$ such that
\[\Big\Vert\sum_{p\in\mathbb{P}_B\cap[N]}\frac{a_p}{N}e^{2\pi ip \xi} \Big\Vert_{L^r(\mathbb{T})}\le C N^{-1/r}\bigg(\sum_{p\in \mathbb{P}_B\cap [N]}\frac{|a_n|^2\psi(p)}{N\log(p)}\bigg)^{1/2}
\]
thus
\begin{equation}\label{eqmaj}
\Big\Vert\sum_{p\in \mathbb{P}_B\cap[N]}a_pe^{2\pi i \xi} \Big\Vert_{L^r(\mathbb{T})}\le C N^{-1/r}\bigg(N\sum_{p\in \mathbb{P}_B\cap [N]}\frac{\psi(p)}{\log(p)}\bigg)^{1/2}\le CN^{-1/r}\bigg(N\sum_{p\in \mathbb{P}_B\cap [N]}\frac{\varphi_2'(p)}{\log(p)}\bigg)^{1/2}
\end{equation}
On the other hand we have that
\[
\bigg\Vert\sum_{p\in \mathbb{P}_B\cap[N]} e^{2 \pi i p \xi} \bigg\Vert_{L^r(\mathbb{T})}\ge \bigg(\int_{\frac{-1}{100N}}^{\frac{1}{100N}}\Big|\sum_{p\in \mathbb{P}_B\cap[N]}e^{2\pi i p\xi}\Big|^r\bigg)^{1/r} \gtrsim \frac{|\mathbb{P}_B\cap[N]|}{N^{1/r}}\gtrsim N^{-1/r}\frac{\varphi_2(N)}{\log(N)}
\]
Finally, we will estimate the sum in $\ref{eqmaj}$ using summation by parts together with our asymptotic formula for $\pi_B(N;1,0)$. Let $\varepsilon=\frac{2\gamma_2-1}{8\gamma_2}>0$, then
\[
\sum_{p\in \mathbb{P}_B\cap [N]}\frac{\varphi_2'(p)}{\log(p)}=\pi_B(N;1,0)\frac{\varphi_2'(N)}{\log(N)}-\int_2^N\pi_B(t;1,0)\frac{\varphi_2''(t)\log(t)-\varphi_2'(t)/t}{\log^2(t)}dt\lesssim
\]
\[
\frac{\varphi_2(N)\varphi_2'(N)}{\log^2(N)}+\int_2^N\frac{\varphi_2(t)}{\log(t)}\frac{|t^2\varphi_2''(t)\log(t)-t\varphi_2'(t)|}{t^2\log^2(t)}dt\lesssim \frac{\varphi_2^2(N)}{N\log^2(N)}+\int_{2}^N\frac{\varphi_2^2(t)}{t^2\log^2(t)}dt=
\]
\[
\frac{\varphi_2^2(N)}{N\log^2(N)}+\int_{2}^{N^{\varepsilon}}\frac{\varphi_2^2(t)}{t^2\log^2(t)}dt+\int_{N^{\varepsilon}}^N\frac{\varphi_2^2(t)}{t^2\log^2(t)}dt\lesssim \frac{\varphi_2^2(N)}{N\log^2(N)}+\frac{\varphi_2^2(N^{\varepsilon})}{\varepsilon\log^2(N)}\int_{2}^{+\infty}\frac{1}{t^2}dt+\int_{N^{\varepsilon}}^N\frac{\varphi_2^2(t)}{t^2\log^2(t)}dt \lesssim_{\varepsilon}
\]
\[
\frac{\varphi_2^2(N)}{N\log^2(N)}+\frac{1}{\log^2(N)}\int_{N^{\varepsilon}}^N\frac{\varphi_2^2(t)}{t^2}dt\lesssim
\frac{\varphi_2^2(N)}{N\log^2(N)}
\]
where we have used the fact that $\varphi_2^2(x^{\varepsilon})\lesssim x^{(2\gamma_2-1)/2}\lesssim \varphi^2_2(x)/x$ and also that $\int_{N^{\varepsilon}}^N\frac{\varphi_2^2(t)}{t^2}dt\lesssim \frac{\varphi_2^2(N)}{N}$. To see this, define $\Phi_2(x)=\varphi_2^2(x)/x=x^{2\gamma_2-1}\ell^2_{\varphi_2}(x)$ and notice that one can easily show that $\Phi_2'(x)x\simeq \Phi_2(x) $. Therefore, we may write
\[
\int_{N^{\varepsilon}}^N\frac{\varphi_2^2(t)}{t^2}dt=\int_{N^{\varepsilon}}^N\frac{\Phi_2(t)}{t}dt\lesssim\int_{N^{\varepsilon}}^N\Phi_2'(t)dt\lesssim \Phi_2(N)=  \frac{\varphi_2^2(N)}{N}
\]
This concludes the proof since we have shown that
\[
\Big\Vert\sum_{p\in \mathbb{P}_B\cap[N]}a_pe^{2\pi i \xi} \Big\Vert_{L^r(\mathbb{T})}\lesssim N^{-1/r}\frac{\varphi_2(N)}{\log(N)}\lesssim \bigg\Vert\sum_{p\in \mathbb{P}_B\cap[N]} e^{2 \pi i p \xi} \bigg\Vert_{L^r(\mathbb{T})}
\]
\end{proof}

We wish to finish this section by making some remarks about the restriction Theorem for the much simpler case of the sets $B$. The rather technical Lemma~$\ref{PR2}$ is replaced by  the following.

\begin{lemma}\label{PR1} Let $c_1\in [1,2)$ and $c_2\in [1,6/5) $ and let $\gamma_1=1/c_1$ and $\gamma_2=1/c_2$. Assume we have fixed $b,m\in \mathbb{Z}$ such that $0\le b\le m-1$. Then for every $\chi>0$ such that $(1-\gamma_1)+3(1-\gamma_2)+6\chi<1$ there exists $\chi'>0$ such that

	\begin{equation}\label{EXT}
		\sum_{\substack{n\in B\cap[N] \\ n\equiv b\Mod m}} \psi(n)^{-1} e^{2\pi i n\xi}= \sum_{\substack{n\in [N] \\ n\equiv b \Mod m}} e^{2\pi i n\xi} +O(N^{1-\chi-\chi'})
	\end{equation}
where the implied constant does not depend on $\xi,N,b,m$.
	
\end{lemma}

\begin{proof}
	Use Lemma 3.2 from \cite{HLMP}, together with the identity $
	1_{\{n\in\mathbb{Z}\colon\,n\equiv b\Mod m\}}(k)=\frac{1}{m}\sum_{s=0}^{m-1}e^{2 \pi i s(k-b)/m}$.
\end{proof}
The analogues for $\mathcal{P}_{b,m,N}$ and $\rho_{b,m,N}$ are $\mathcal{M}_{b,m,N}=\{\,n\in[N]:nm+b\in B\,\}$ and
	\[
	\mu_{b,m,N}(n)=\left\{
	\begin{array}{ll}
		\frac{\psi(nm+b)^{-1}}{N}, n\in \mathcal{M}_{b,m,N} \\
		0, n\notin \mathcal{M}_{b,m,N}
	\end{array} 
	\right.\]
Finally, the restriction Theorem is the following
\begin{theorem}[Restriction Theorem for $B$]\label{REST} Let $c_1\in [1,2)$ and $c_2\in [1,6/5)$ and assume we have fixed $h_1$, $h_2$, $\psi$ and $B$ as in the introduction. Let $N,b,m\in \mathbb{Z}$ be such that $0\le b< m\le \log(N)$. For each $r> 2+\frac{12-12/c_2}{1/c_1+3/c_2-3}$, there exists a constant $C=C(r,h_1,h_2,\psi)>0$ such that 
	\[
	||{S_{b,m,N}}(f)||_{L^r(\mathbb{T})}\le C N^{-1/r} ||f||_{L^2(\mathcal{M}_{b,m,N},\mu_{b,m,N})}
	\]
	
	for all $f\in L^2(\mathcal{M}_{b,m,N},\mu_{b,m,N})$, where ${S_{b,m,N}}: \mathcal{C}(\mathcal{M}_{b,m,N})\to \mathcal{C}(\mathbb{T})$ is such that
	\[
	{S_{b,m,N}}(f)(\xi)=\mathcal{F}[f\mu_{b,m,N}](\xi)=\sum_{n\in\mathbb{Z}}f(n)\mu_{b,m,N}(n)e^{2\pi i n\xi}\text{ for all }\xi\in\mathbb{T}
	\]
		
\end{theorem}

\begin{proof}
This Theorem is a generalization of Proposition~3.1 in \cite{HLMP}. A similar argument to the one presented there works here as well. Essentially a $TT^*$ argument and interpolation are the key ingredients of the proof, similarly to the proof of Theorem~$\ref{RestOur}$, but much simpler.
\end{proof}

\begin{remark}
Let's remark that in the same spirit as in the proof of Theorem~$\ref{HLMPPB}$, this restriction theorem implies that the set $B$ has the Hardy--Littlewood majorant property. For the specific formulation and proof we point the reader to \cite{HLMP}.
\end{remark}
Finally, we wish to comment that the restriction Theorem for the set $B$ together with an appropriate Transference Principle, analogous to the one we present in the next section for the sets $\mathbb{P}_B$, are  sufficient to yield Theorem~$\ref{T1}$.
\section{Transference Principle}\label{trpr} We are now ready to prove Theorem~$\ref{T2}$. We fix $c_1$, $c_2\in [1,95/94)$, $h_1$, $h_2$, $\psi$ and $B$ as in the introduction. This implies that there exists $\chi>0$ such that $16(1-\gamma_1)+17(1-\gamma_2)+31\chi<1$, and therefore according to Proposition $\ref{PR2}$, there exists a real number $\chi'>0$ such that the estimate in $\ref{EXT2}$ holds. Throughout the discussion here we have fixed such $\chi,\chi'>0$. All the implied constants in our work in this section may depend on $h_1,h_2,\psi,\chi,\chi'$ and on nothing else unless we explicitly indicate it. We transfer our problem to $\mathbb{Z}_N=\mathbb{Z}/N\mathbb{Z}$.
\begin{lemma}\label{TPrL1}
	Let $A_0\subseteq \mathbb{P}\cap B=\mathbb{P}_B$ and assume that
	\[
	\limsup_{N\to\infty}\frac{|A_0\cap [N]|}{|\mathbb{P}_B\cap [N]|}>0\text{ or equivalently }\limsup_{N\to\infty}\frac{\log(N)|A_0\cap [N]|}{\varphi_2(N)} >0
	\]
	then
	\[
	\limsup_{N\to\infty}\frac{\log(N)|A_0\cap [N,2N]|}{\varphi_2(N)}>0
	\]
\end{lemma}

\begin{proof}
	We know that $|\mathbb{P}_B\cap [N]|\simeq\varphi_2(N)/\log(N)$ and thus there exists a positive constant $C$ and a natural number $N_1$ such that $|\mathbb{P}_B\cap [N]|\le C\frac{\varphi_2(N)}{\log(N)}$ for all $N\ge N_1$. We also have that there exists a positive real number $\alpha_0$ for which there are infinitely many naturals numbers $N$ such that $\frac{\log(N)|A_0\cap [N]|}{\varphi_2(N)}>\alpha_0$. We have that $\varphi_2(x)=x^{\gamma_2}\ell_{\varphi_2}(x)$ and for all real numbers $t>0$ 	we have that $\ell_{\varphi_2}(tx)\sim\ell_{\varphi_2}(x)\text{ as }x\to\infty$, see Lemma 2.6 in \cite{MMR}, page 6. Let's fix a real number $t=2^{-k}$ for some $k\in\mathbb{N}$ such that $t^{\gamma_2}<\frac{\alpha_0}{8C}$. We will have that
	\[
	\varphi_2(tN)=t^{\gamma_2}N^{\gamma_2}\ell_{\varphi_2}(tN)=t^{\gamma_2}\varphi_2(N)\frac{\ell_{\varphi_2}(tN)}{\ell_{\varphi_2}(N)}=t^{\gamma_2}\varphi_2(N)+\left(\frac{\ell_{\varphi_2}(tN)}{\ell_{\varphi_2}(N)}-1 \right)t^{\gamma_2}\varphi_2(N)
	\]
	Thus there exists a natural number $N_2$ such that for all $N\ge N_2$ we have that $\varphi_2(tN)\le 2t^{\gamma_2}\varphi_2(N)$. Let's notice that for all natural numbers $N$ such that $N\ge\max\{2N_1/t,N_2\}$ and such that $\frac{|A_0\cap [N]}{\varphi_2(N)}>\alpha_0$ we have
	\[
	|A_0\cap[tN,N]=|A_0\cap[N]|-|A_0\cap[1,tN)|\ge |A_0\cap[N]|-|\mathbb{P}_B\cap[1,tN]|\ge 
	\]
	\[
	\alpha_0\frac{\varphi_2(N)}{\log(N)}-C\frac{\varphi_2(tN)}{\log(N)}
	\ge \alpha_0\frac{\varphi_2(N)}{\log(N)}-2Ct^{\gamma_2}\frac{\varphi_2(N)}{\log(N)}
	\ge \alpha_0\frac{\varphi_2(N)}{\log(N)}-\frac{2\alpha_0}{8}\frac{\varphi_2(N)}{\log(N)}=\frac{3\alpha_0}{4}\frac{\varphi_2(N)}{\log(N)}
	\]
	We have that $k=\log_2(1/t)$ and
	
	\[
	\sum_{l=1}^k|A_0\cap[2^{l-1}tN,2^ltN]|\ge \sum_{l=1}^{k-1}|A_0\cap[2^{l-1}tN,2^ltN)|+|A_0\cap[2^{k-1}tN,2^ktN]|=|A_0\cap[tN,N]|\ge\frac{3\alpha_0}{4}\frac{\varphi_2(N)}{\log(N)}\]
Thus there exists a natural number $l\in[1,k]$ such that $|A_0\cap[2^{l-1}tN,2^ltN]|\ge \frac{3\alpha_0}{4k}\frac{\varphi_2(N)}{\log(N)}$. Since $\varphi_2$ is increasing, we have
	\[
	\frac{\log(2^{l-1}tN)|A_0\cap[2^{l-1}tN,2^ltN]|}{\varphi_2(2^{l-1}tN)}\ge \frac{3\alpha_0 }{4k}\frac{\log(2^{l-1}tN)\varphi_2(N)}{\log(N)\varphi_2(2^{l-1}tN)}\ge C_k\alpha_0
	\]
We note that $t=2^{-k}$ is fixed, $\varphi_2(2x)\lesssim \varphi_2(x)$ and that the previous inequality holds for infinitely many natural numbers. This gives that there exists $\alpha_0'>0$ such that
	\[
	\limsup_{N\to\infty}\frac{\log(N)|A_0\cap [N,2N]|}{\varphi_2(N)}>\alpha_0'
	\]
which is the desired result.
\end{proof}

\begin{lemma}\label{TPrL2}
	Assume $A_0\subseteq \mathbb{P}_B$ has positive upper relative density and thus, according to the previous lemma, there exists a positive real number $\alpha_0$ such that $\limsup_{n\to \infty} \frac{\log(n)|A_0\cap[n,2n]|}{\varphi_2(n)}>\alpha_0$. If $A_0$ does not contain 3APs, then there exists a small positive number $\alpha$ and infinitely many prime numbers $N$ with the property that for each such number there exists a set $A=A_N\subseteq \{1,2,\dotsc\, \lfloor N/2 \rfloor\}$ and an integer $W=W_N\in[1/8\log\log N,1/2\log\log N]$ such that
	\begin{itemize}
		\item[i)] $A=A_N$ has no 3APs,
		\item[ii)] $\rho_{b,m,N}(A)\ge \alpha$ for some $b\in\{0,\dotsc,m-1\}$, with $(b,m)=1$, where $m=\prod_{p\in \mathbb{P}\cap[W]}p$.
	\end{itemize}
\end{lemma}

\begin{proof}
	Since $\limsup_{n\to \infty} \frac{\log(n)|A_0\cap[n,2n]|}{\varphi_2(n)}>\alpha_0$, we will have infinitely many even numbers $n\in\mathbb{N}$ such that $|A_0\cap[n/2,n]|\ge \frac{\alpha_0\varphi_2(n/2)}{2\log(n/2)}\gtrsim\frac{\alpha_0\varphi_2(n)}{\log(n)}$. Let $W=\lfloor1/4\log\log(n)\rfloor$ and $m=\prod_{p\in\mathbb{P}\cap[W]}p$, and notice that $m=\prod_{p\in\mathbb{P}\cap[W]}p\le 4^W\le e^{2/4\log\log(n)}=\log^{1/2}(n)$. According to Bertrand’s postulate, we know that there exists a prime number $N\in[2n/m,4n/m]$. We have that $W\in[1/8\log\log(N),1/2\log\log(N)]$ and also that
	\[
	\sum_{\substack{b\in\{0,\dotsc,m-1\}\\(b,m)=1}}\sum_{k=n/2}^n 1_{A_0\cap P_{b,m}}(k)= |A_0\cap[n/2,n]|\gtrsim \frac{\alpha_0\varphi_2(n)}{\log(n)}\text{ where }P_{b,m}=\{n\in\mathbb{Z}\colon\,n\equiv b\Mod m\}
	\]
We have $\psi(x)\simeq \varphi_2'(x)  \simeq \varphi_2(x)/x$ and $\varphi_2(x)\simeq \varphi_2(2x)$, and thus we get
	\[
	\sum_{\substack{b\in\{0,\dotsc,m-1\}\\(b,m)=1}}\sum_{k=n/2}^n 1_{A_0\cap P_{b,m}}(k)\psi(k)^{-1}\log(k)\gtrsim \sum_{\substack{b\in\{0,\dotsc,m-1\}\\(b,m)=1}}\sum_{k=n/2}^n 1_{A_0\cap P_{b,m}}(k)\varphi_2'(k)^{-1}\log(k) \gtrsim
	\]
	\[
	\log(n)\varphi_2'(n)^{-1}\sum_{\substack{b\in\{0,\dotsc,m-1\}\\(b,m)=1}}\sum_{k=n/2}^n 1_{A_0\cap P_{b,m}}(k)\gtrsim n\log(n)\varphi_2(n)^{-1}\sum_{\substack{b\in\{0,\dotsc,m-1\}\\(b,m)=1}}\sum_{k=n/2}^n 1_{A_0\cap P_{b,m}}(k) \gtrsim \alpha_0 n
	\]
	By the pigeonhole principle there exists $b\in\{0,\dotsc,m-1\}$ with $(b,m)=1$ and such that 
	\[
	\sum_{k=n/2}^n 1_{A_0\cap P_{b,m}}(k)\psi(k)^{-1}\log(k)\gtrsim \alpha_0 n/\phi(m)
	\]
	Let $A=A_N=\frac{1}{m}(A_0\cap P_{b,m}\cap\{n/2,\dotsc,n\}-b)$ and notice that $A\subseteq \{1,\dotsc,\lfloor N/2\rfloor\}$. Since $A_0$ does not have any 3APs, neither will $A$, and notice that this means that it will not have such progressions even when considered as a subset of $\mathbb{Z}_N$. Finally, notice that $A\subseteq \mathcal{P}_{b,m,N}$ and with a change of variables we get
	\begin{equation}\label{FP2}
		\rho_{b,m,N}(A)=\sum_{l\in A}\rho_{b,m,N}(l)=\sum_{k=n/2}^n 1_{A_0\cap P_{b,m}}(k)\psi(k)^{-1}\frac{\phi(m)\log(k)}{mN}\gtrsim \alpha_0 \frac{n}{mN}\ge \alpha_0/4
	\end{equation}
\end{proof}
From now on we fix $A_0\subseteq \mathbb{P}_B$ with positive upper relative density and we assume for the sake of a contradiction that it does not contain any 3APs. We see that Lemmas $\ref{TPrL1}$ and $\ref{TPrL2}$ are applicable.
\begin{lemma}\label{Allbut0}
	Let $N\in\mathbb{P}$, $W\in[1/8\log\log N,1/2\log\log N]$, $m$ and $b$ be the integers of the previous lemma. Then for sufficiently large $N$ we get
	\[
	\sup_{\xi\in\mathbb{Z}_N\setminus\{0\}} \left| \mathcal{F}_{\mathbb{Z}_N}[\rho_{b,m,N}](\xi) \right|\lesssim \log\log W/W
	\]
\end{lemma}
\begin{proof}
	We will use the fact that for sufficiently large $N$ we have that
	\[
	\sup_{\xi\in\mathbb{Z}_N}\big|\mathcal{F}_{\mathbb{Z}_N}[\lambda_{b,m,N}](\xi)\big|\le 2\log\log W/W
	\]
	which has been established in Green's work, see \cite{BenGreen}, Lemma 6.2, page 17, together with our estimate
	\[
	\sum_{\substack{n\in \mathbb{P}\cap B_N \\ n\equiv b\Mod m}} \psi(n)^{-1}\log(n) e(n\xi)= \sum_{\substack{n\in \mathbb{P}\cap[N] \\ n\equiv b \Mod m}} \log(n)e(n\xi) +O(N^{1-\chi-\chi'})\text{, where }e(x)=e^{2 \pi i x}
	\]
	We have
	\[
	\sup_{\xi\in\mathbb{Z}_N\setminus\{0\}} \left| \mathcal{F}_{\mathbb{Z}_N}[\rho_{b,m,N}](\xi) \right|\le \sup_{\xi\in\mathbb{Z}_N\setminus\{0\}} \left| \mathcal{F}_{\mathbb{Z}_N}[\rho_{b,m,N}](\xi)-\mathcal{F}_{\mathbb{Z}_N}[\lambda_{b,m,N}](\xi) \right|+\sup_{\xi\in\mathbb{Z}_N\setminus\{0\}} \left| \mathcal{F}_{\mathbb{Z}_N}[\lambda_{b,m,N}](\xi) \right| \le
	\]
	\[
	\sup_{\xi\in\mathcal{F}_{\mathbb{Z}_N}\setminus\{0\}}\bigg|\sum_{\substack{n\in\{1,\dotsc,N\}\\nm+b\in\mathbb{P}_B }}\frac{\phi(m)\log(mn+b)}{\psi(mn+b)mN}e^{-2 \pi i n\xi/N}-\sum_{\substack{n
			\in\{1,\dotsc,N\}\\nm+b\in\mathbb{P}}} \frac{\phi(m)\log(mn+b)}{mN}e^{-2 \pi i n\xi/N}\bigg| +
	2\log\log W/W=
	\]
	\[
	\sup_{\xi\in\mathcal{F}_{\mathbb{Z}_N}\setminus\{0\}}\bigg|\sum_{\substack{n\in\{1,\dotsc,N\}\\nm+b\in\mathbb{P}_B }}\frac{\phi(m)\log(mn+b)}{\psi(mn+b)mN}e^{\frac{-2 \pi i (nm+b)\xi}{Nm}}-\sum_{\substack{n
			\in\{1,\dotsc,N\}\\nm+b\in\mathbb{P}}} \frac{\phi(m)\log(mn+b)}{mN}e^{\frac{-2 \pi i (nm+b)\xi}{Nm}}\bigg| +
	2\log\log W/W\lesssim
	\]
	\[
	\sup_{\xi\in\mathcal{F}_{\mathbb{Z}_N}\setminus\{0\}}\bigg|\sum_{\substack{k\in\mathbb{P}\cap B_{mN+b}\\k\equiv b\Mod{m} }}\frac{\phi(m)\log(k)}{\psi(k)mN}e^{\frac{-2 \pi i k\xi}{Nm}}-\sum_{\substack{n\in \mathbb{P}\cap[mN+b] \\ n\equiv b \Mod m}} \frac{\phi(m)\log(k)}{mN}e^{\frac{-2 \pi i k\xi}{Nm}}\bigg| +N^{-1}\log(N)+
	2\log\log W/W\lesssim
	\]
	\[
	\frac{(mN+b)^{1-\chi-\chi'}}{N}+2\log\log W/W\lesssim N^{-\chi}+2\log\log W/W\lesssim \log\log(N)
	\]
\end{proof}
We define a new measure on $\mathbb{Z}_N$ by letting
$
a(S)=\sum_{k\in S} 1_{S\cap A}(k) \rho_{b,m,N}(k)$ for any $S\subseteq \mathbb{Z}_N
$, where we are considering $\rho_{b,m,N}$ as a function on $\mathbb{Z}_N$ in the obvious way.
According to Lemma $\ref{TPrL2}$, we will have that $a(\mathbb{Z}_N)\ge \alpha$. Now, we define yet another measure on $\mathbb{Z}_N$. Let $\delta,\varepsilon\in(0,1)$ be numbers that will be chosen later and define
\[
R=\{\,\xi\in\mathbb{Z}_N:\,\,|\mathcal{F}_{\mathbb{Z}_N}[a](\xi)|\ge \delta\,\}
\]
If $R=\{\xi_1,\dotsc,\xi_k\}$ with $|R|=k$, then define
\[
B=\{\,x\in\mathbb{Z}_N:\,\,\sup_{i\in [k]}\Big\|\frac{x\xi_i}{N}\Big\| \le\varepsilon\,\}\text{, where } \|x\|=\min\{|x-n|\colon n\in\mathbb{Z}\}
\]
We will have that $|B|\ge \varepsilon^k N$, see Lemma 4.20 in \cite{TV}, page 166, and thus $B$ is non-empty for sufficiently big $N$. Let $\beta(y)=\frac{1}{|B|}1_B(y)$ for all $y\in\mathbb{Z}_N$. Finally, let $a_1=a*\beta*\beta$, and note that $a_1(\mathbb{Z}_N)\ge \alpha$.
\begin{lemma}\label{L4HL3}
	Let $N\in\mathbb{P}$, $W\in[1/8\log\log N,1/2\log\log N]$, $m$ and $b$ be the integers as before and assume that $\varepsilon^k\ge \log\log W/W$, then
	\[
	||a_1||_{\ell^{\infty}(\mathbb{Z}_N)}\lesssim 1/N
	\]
\end{lemma}
\begin{proof}
	Here we use the Fourier Inversion Formula together with Lemma~$\ref{Allbut0}$ to obtain
	\[
	a_1(x)=a*\beta*\beta(x)\le \rho_{b,m,N}*\beta*\beta(x)=\frac{1}{N}\mathcal{F}_{\mathbb{Z}_N}^{-1}[\mathcal{F}_{\mathbb{Z}_N}[\rho_{b,m,N}*\beta*\beta]](x)=
	\]
	\[
	\frac{1}{N}\sum_{\xi\in\mathbb{Z}_N}\mathcal{F}_{\mathbb{Z}_N}[\rho_{b,m,N}](\xi)\mathcal{F}_{\mathbb{Z}_N}[\beta](\xi)^2e^{\frac{2 \pi i \xi x}{N}}\le
	\]
	\[
	\frac{1}{N}\mathcal{F}_{\mathbb{Z}_N}[\rho_{b,m,N}](0)\mathcal{F}_{\mathbb{Z}_N}[\beta](0)^2+\frac{1}{N}\sup_{\zeta\in\mathbb{Z}\setminus\{0\}}\big|\mathcal{F}_{\mathbb{Z}_N}[\rho_{b,m,N}](\zeta)\big|\sum_{\xi\in\mathbb{Z}_N\setminus\{0\}}\mathcal{F}_{\mathbb{Z}_N}[\beta](\xi)^2
	\]
	It is not difficult to see that $\mathcal{F}_{\mathbb{Z}_N}[\rho_{b,m,N}](0)\lesssim 1$ and $\mathcal{F}_{\mathbb{Z}_N}[\beta](0)=1$. We also have
	
	\[
	\sum_{n\in\mathbb{Z}_N\setminus\{0\}}\left|\mathcal{F}_{\mathbb{Z}_N}[\beta](n)^2\right|\le \sum_{n\in\mathbb{Z}_N}\left(\mathcal{F}_{\mathbb{Z}_N}[\beta](n)\right)\overline{\left(\mathcal{F}_{\mathbb{Z}_N}[\beta](n)\right)}=\sum_{n\in\mathbb{Z}_N}\sum_{m\in\mathbb{Z}_N}\sum_{k\in\mathbb{Z}_N}\beta(m)\beta(k)e^{2\pi in m/N}e^{-2\pi in k/N}=
	\] 
	\[
	=\sum_{n\in\mathbb{Z}_N}\sum_{m\in\mathbb{Z}_N}|\beta(m)|^2+\sum_{m\in\mathbb{Z}_N}\sum_{k\in\mathbb{Z}_N\setminus\{m\}}\beta(m)\beta(k)\sum_{n\in\mathbb{Z}_N}e^{2\pi in m/N}e^{-2\pi in k/N} =N|B|^{-1}
	\]
	Putting everything together, we use the previous Lemma as well as the fact that $|B|\ge \varepsilon^k N \ge N\log\log W/W$ to obtain
	\[
	|a_1(x)|\lesssim 1/N +|B|^{-1}\log\log W/W\lesssim 1/N
	\]
	as desired.
\end{proof}
We now use the Restriction Theorem for the set $\mathbb{P}_B$, to obtain a discrete version of our Restriction Theorem which is called discrete majorant property.

\begin{lemma}[Discrete majorant property]\label{DscMjP}
	Assume that $r>2+\frac{62-62/c_2}{16/c_1+17/c_2-32}$, then there exists a positive constant $C=C(r,h_1,h_2,\psi)$ such that 
	\[
	||\mathcal{F}_{\mathbb{Z}_N}[a]||_{\ell^{r}(\mathbb{Z}_N)}\le C
	\]
\end{lemma}
\begin{proof}
	We will use Theorem~$\ref{RestOur}$ together with Marcinkiewicz–Zygmund theorem, see Lemma 6.5 in \cite{BenGreen}. Let's notice that
	\[
	||\mathcal{F}_{\mathbb{Z}_N}[a]||^r_{\ell^r(\mathbb{Z}_N)}=\sum_{k\in\mathbb{Z}_N}|\mathcal{F}_{\mathbb{Z}_N}[a](k)|^r=\sum_{k=0}^{N-1}\bigg|\sum_{l=1}^Na(l)e^{2 \pi i l \frac{k}{N}}\bigg|^r=\sum_{k=0}^{N-1}\bigg|\mathcal{F}_{\mathbb{Z}}[a](k/N)\bigg|^r\lesssim_{r}
	\]
	\[
	N\int_{\mathbb{T}}\big|\mathcal{F}_{\mathbb{Z}}[a](t)\big|^rdt=N\int_{\mathbb{T}}\big|\mathcal{F}_{\mathbb{Z}}[1_A\rho_{b,m,N}](t)\big|^rdt=N||T^B_{b,m,N}(1_A)||^r_{L^r(\mathbb{T})}\lesssim_{r,h_1,h_2,\psi}||1_A||^r_{L^2(P_{b,m,N},\rho_{b,m,N})}\lesssim 1
	\]
	
\end{proof}
We finish the proof by introducing and estimating certain trilinear forms. Let $\Lambda$ be the following trilinear form
\[
\Lambda(f,g,h)=\sum_{x,d\in\mathbb{Z}_N}f(x)g(x+d)h(x+2d)
\]
where $f,g,h\colon \mathbb{Z}_N\to\mathbb{C}$ are arbitrary functions. By the Fourier Inversion formula the following useful identity is holds whenever $N$ is odd
\[
\Lambda(f,g,h)=\frac{1}{N}\sum_{\xi\in\mathbb{Z}_N}\mathcal{F}_{\mathbb{Z}_N}[f](\xi)\mathcal{F}_{\mathbb{Z}_N}[g](-2\xi)\mathcal{F}_{\mathbb{Z}_N}[h](\xi)
\]
Let's notice that since all the sets $A=A_N$ produced by Lemma $\ref{TPrL2}$ do not contain 3APs, we have that
\[
\Lambda(a,a,a)=\sum_{n\in\mathbb{Z}_N}a(x)^3\le \sum_{n=1}^N\rho_{b,m,N}(n)^3=\sum_{\substack{n\in[N]\\mn+b\in\mathbb{P}_B}}\bigg(\frac{\phi(m)\log(nm+b)}{mN\psi(nm+b)}\bigg)^3\lesssim
\]
\[
\frac{\log^3(Nm+b)}{N^3}\sum_{\substack{n\in[N]\\mn+b\in\mathbb{P}_B}}\frac{1}{\varphi_2'(nm+b)^3}\le\frac{N\log^3(Nm+b)}{N^3\varphi_2'(Nm+b)^3} \lesssim \frac{m^3N\log^3(Nm)}{\varphi_2^3(Nm)^3}\lesssim \frac{N\log^3(N)\log^3(N\log(N))}{\varphi_2^3(N\log(N))}\lesssim_{\varepsilon_1}
\]
\[
\frac{N\log^6(N)}{(N\log(N))^{3\gamma_2-3\varepsilon_1}}\lesssim_{\varepsilon_1} N^{1-3\gamma_2+4\varepsilon_1}\le N^{-3/2}
\]
where we have completed the estimates by using the basic properties of $\varphi_2$, see Lemma~2.6 in \cite{MMR}, and by choosing a positive number $\varepsilon_1<\frac{3\gamma_2-5/2}{4}$, which is possible for $\gamma_2\in (94/95,1]$.
\begin{lemma}\label{UBound}
	For any $r>2+\frac{62-62/c_2}{16/c_1+17/c_2-32}$, there exists a positive constant $C_1=C_1(r,h_1,h_2,\psi)$ such that 
	\begin{equation}
		\Lambda(a_1,a_1,a_1)\le C_1 N^{-3/2}+C_1N^{-1}(\varepsilon^2\delta^{-r}+\delta^{2-r/r'})
	\end{equation}
\end{lemma}
\begin{proof}
	According to the previous estimate, we have that there exists a constant $C_0=C_0(h_1,h_2,\psi)$ such that $\Lambda(a,a,a)\le C_0N^{-3/2}$. Therefore, we have that 
	\[
	\Lambda(a_1,a_1,a_1)\le\Lambda(a_1,a_1,a_1)-\Lambda(a,a,a)+C_0N^{-3/2}=
	\]
	
	\[
	\frac{1}{N}\sum_{\xi\in\mathbb{Z}_N}\mathcal{F}_{\mathbb{Z}_N}[a_1](\xi)\mathcal{F}_{\mathbb{Z}_N}[a_1](-2\xi)\mathcal{F}_{\mathbb{Z}_N}[a_1](\xi)-\frac{1}{N}\sum_{\xi\in\mathbb{Z}_N}\mathcal{F}_{\mathbb{Z}_N}[a](\xi)\mathcal{F}_{\mathbb{Z}_N}[a](-2\xi)\mathcal{F}_{\mathbb{Z}_N}[a](\xi)+C_0N^{-3/2}=
	\]
	
	\[
	\frac{1}{N}\sum_{\xi\in\mathbb{Z}_N}\mathcal{F}_{\mathbb{Z}_N}[a](\xi)^2\mathcal{F}_{\mathbb{Z}_N}[\beta](\xi)^4\mathcal{F}_{\mathbb{Z}_N}[a](-2\xi)\mathcal{F}_{\mathbb{Z}_N}[\beta](-2\xi)^2-\frac{1}{N}\sum_{\xi\in\mathbb{Z}_N}\mathcal{F}_{\mathbb{Z}_N}[a](\xi)^2\mathcal{F}_{\mathbb{Z}_N}[a](-2\xi)+C_0N^{-3/2}=
	\]
	
	\[
	\frac{1}{N}\sum_{\xi\in\mathbb{Z}_N}\mathcal{F}_{\mathbb{Z}_N}[a](\xi)^2\mathcal{F}_{\mathbb{Z}_N}[a](-2\xi)\left(\mathcal{F}_{\mathbb{Z}_N}[\beta](\xi)^4\mathcal{F}_{\mathbb{Z}_N}[\beta](-2\xi)^2-1\right)+C_0N^{-3/2}
	\]
	For every $\xi\in R$, we will have that $\left|\mathcal{F}_{\mathbb{Z}_N}[\beta](\xi)^4\mathcal{F}_{\mathbb{Z}_N}[\beta](-2\xi)^2-1\right|\le 2^{12}\varepsilon^2$; the proof is straightforward and can be found in \cite{BenGreen}, see Lemma~6.7, page 19. On the one hand, Lemma~$\ref{DscMjP}$ suggests that there exists a constant $C=C(r,h_1,h_2,\psi)$ such that 
\[
	||\mathcal{F}_{\mathbb{Z}_N}[a]||_{\ell^{\infty}(\mathbb{Z}_N)}\le ||\mathcal{F}_{\mathbb{Z}_N}[a]||_{\ell^r(\mathbb{Z}_N)}\le C
	\]
	On the other hand, we have
	
	\begin{equation}\label{lastarg}
		\delta^r|R|\le \sum_{\xi\in R}|\mathcal{F}_{\mathbb{Z}_N}[a](\xi)|^r\le \sum_{\xi\in \mathbb{Z}_N}|\mathcal{F}_{\mathbb{Z}_N}[a](\xi)|^r\le C^r
	\end{equation}
	Thus
	
	\[
	\left|\sum_{\xi\in R}\mathcal{F}_{\mathbb{Z}_N}[a](\xi)^2\mathcal{F}_{\mathbb{Z}_N}[a](-2\xi)\left(\mathcal{F}_{\mathbb{Z}_N}[\beta](\xi)^4\mathcal{F}_{\mathbb{Z}_N}[\beta](-2\xi)^2-1\right)\right|\le
	\]
	
	\[
	2^{12}\varepsilon^2\sum_{\xi\in R}|\mathcal{F}_{\mathbb{Z}_N}[a](\xi)|^2|\mathcal{F}_{\mathbb{Z}_N}[a](-2\xi)|\le 2^{12}\varepsilon^2  C^3 |R|\le C^{3+r} 2^{12}\varepsilon^2\delta^{-r}  
	\]
Now we bound the sum along $\xi\notin R$. Firstly, we note that $\sup_{\xi\in\mathbb{Z}_N}\left|\mathcal{F}_{\mathbb{Z}_N}[\beta](\xi)^4\mathcal{F}_{\mathbb{Z}_N}[\beta](-2\xi)^2-1\right|\le 2$ since $|\mathcal{F}_{\mathbb{Z}_N}[\beta](\xi)|\le1$. For now, let's assume that $r\in (2+\frac{62-62/c_2}{16/c_1+17/c_2-32},3)$, which is possible since $c_1\in [1,95/94)$, $c_2\in[1,95/94)$ and thus $0\le\frac{62-62/c_2}{16/c_1+17/c_2-32}< 1$. Let $r'$ be such that $1/r+1/r'=1$ and note that $2-r/r'>0$. Thus
	
	\[
	\left|\sum_{\xi\notin R}\mathcal{F}_{\mathbb{Z}_N}[a](\xi)^2\mathcal{F}_{\mathbb{Z}_N}[a](-2\xi)\left(\mathcal{F}_{\mathbb{Z}_N}[\beta](\xi)^4\mathcal{F}_{\mathbb{Z}_N}[\beta](-2\xi)^2-1\right)\right|\le
	\]
	
	\[
	2\sum_{\xi\notin R}|\mathcal{F}_{\mathbb{Z}_N}[a](\xi)^2\mathcal{F}_{\mathbb{Z}_N}[a](-2\xi)|\le 2\sup_{\xi\notin R}\big|\mathcal{F}_{\mathbb{Z}_N}[a](\xi)\big|^{2-r/r'}\sum_{\xi\notin R}|\mathcal{F}_{\mathbb{Z}_N}[a](\xi)^{r/r'}\mathcal{F}_{\mathbb{Z}_N}[a](-2\xi)|\le
	\]
	
	\[
	2\delta^{2-r/r'}\big( \sum_{\xi\in\mathbb{Z}_N}|\mathcal{F}_{\mathbb{Z}_N}[a](\xi)|^{r}\big)^{1/r'}\big( \sum_{\xi\in\mathbb{Z}_N}|\mathcal{F}_{\mathbb{Z}_N}[a](-2\xi)|^{r}\big)^{1/r}=2\delta^{2-r/r'}\sum_{\xi\in\mathbb{Z}_N}|\mathcal{F}_{\mathbb{Z}_N}[a](\xi)|^{r}\le 2C^r\delta^{2-r/r'} 
	\]
	For $C_1=\max\{\,C_0, C^{3+r} 2^{12},2C^r\,\}$ we get
	
	\[
	\Lambda(a_1,a_1,a_1)\le C_1N^{-3/2}+C_1N^{-1}(\varepsilon^2\delta^{-r}+\delta^{2-r/r'})
	\]
	For $r\ge3$, let $s\in (2+\frac{12-12/c_2}{1/c_1+3/c_2-3},3)$, the previous argument yields
	
	\[
	\Lambda(a_1,a_1,a_1)\le C_1N^{-3/2}+C_1N^{-1}(\epsilon^2\delta^{-r}+\delta^{2-s/s'})
	\]  and since $\delta^{2-s/s'}\le\delta^{2-r/r'}$ we get the desired result.
\end{proof}
The following Lemma provides a lower bound for $\Lambda(a_1,a_1,a_1)$ and similarly to the work of Green \cite{BenGreen}, we adapt Varnavides' argument  \cite{VARN} in order to attain it.
\begin{lemma}\label{LBound}
	There exist positive constants $C_2=C_2(h_1,h_2,\psi)$, $C_3=C_3(h_1,h_2,\psi)$ such that 
	\begin{equation}
		\Lambda(a_1,a_1,a_1)\ge C_2N^{-1}e^{-C_3\alpha^{-1}\log(1/\alpha)}
	\end{equation}
\end{lemma}
\begin{proof}
	We know from Sanders' result \cite{SA1} that there exists a positive constant $D$ such that if $M\ge e^{D\alpha^{-1}\log^5(1/\alpha)}$ then all subsets of $[M]$ with density at least $\frac{\alpha}{4C}$ contain a non--trivial three--term arithmetic progression, where $C=C(h_1,h_2,\psi)$ is the implied constant appearing in Lemma $\ref{L4HL3}$ and without loss of generality let us assume that $C\ge2$. Let $A'=\{\,x\in\mathbb{Z}_N:a_1(x)\ge \frac{\alpha}{NC}\,\}$ and note that from Lemma \ref{L4HL3}
	\[
	\alpha\le a_1(\mathbb{Z}_N)=\sum_{x\in A'}a_1(x)+\sum_{x\notin A'}a_1(x)\le \frac{|A'|C}{N} +\frac{(N-|A'|)\alpha}{NC}
	\] which implies that $|A'|\ge \frac{\alpha N}{2C}$. Let's define $Z=|\{\,(x,d)\in\mathbb{Z}_N^2: x,x+d,x+2d\in A'\,\}|$ to be the number of arithmetic progressions of length three in $A'$. Note that
	\[
	\Lambda(a_1,a_1,a_1)=\sum_{x,d\in\mathbb{Z}_N}a_1(x)a_1(x+d)a_1(x+2d)\ge Z\frac{\alpha^3}{N^3C^3}
	\] 
	Let us fix $M=\Big\lceil e^{D\alpha^{-1}\log^5(1/\alpha)} \Big\rceil$. For $N\ge M$, we find a lower bound for $Z$. Let $a,d\in\mathbb{Z}_N$ be such that $d\neq 0$ and let $P_{a,d}=\{\,a,a+d,a+2d,\dotsc,a+(M-1)d\}$. For a fixed $d$, Sanders' result gives that for every $a\in\mathbb{Z}_N$ such that $|A'\cap P_{a,d}|/M\ge \alpha/(4C) $ we must necessarily have that $A'\cap P_{a,d}$  contains a non-trivial three-term arithmetic progression. We have
	\[
	\sum_{a\in\mathbb{Z}_N}|A'\cap P_{a,d}|=M|A'|\ge \frac{\alpha  M N}{2C}
	\] 
	where we have used the fact that in the sum we count every element of $A'$ exactly $M$ times. Therefore
	\[
	\frac{\alpha  M N}{2C}\le \sum_{a\in\mathbb{Z}_N:|A'\cap P_{a,d}|\ge \alpha M/(4C)}  |A'\cap P_{a,d}|+\sum_{a\in\mathbb{Z}_N:|A'\cap P_{a,d}|< \alpha M/(4C)}  |A'\cap P_{a,d}|
	\]
	and we have
	\[
	\sum_{a\in\mathbb{Z}_N:|A'\cap P_{a,d}|< \alpha M/(4C)}  |A'\cap P_{a,d}|<\frac{ \alpha MN}{4C}
	\]
	which in turn implies that
	\[
	\frac{ \alpha MN}{4C}\le \sum_{a\in\mathbb{Z}_N:|A'\cap P_{a,d}|\ge \alpha M/(4C)}  |A'\cap P_{a,d}|\le |\{\,a\in\mathbb{Z}_N:|A'\cap P_{a,d}|\ge \alpha M/(4C)\,\}|M  
	\]
	Thus $\frac{ \alpha N}{4C}\le |\{\,a\in\mathbb{Z}_N:|A'\cap P_{a,d}|\ge \alpha M/(4C)\,\}|$. There are at least $ \frac{ \alpha N}{4C}$ values of $a\in\mathbb{Z}_N$ such that $|A'\cap P_{a,d}|\ge \alpha M/(4C)$. Thus, for each $d\in\mathbb{Z}_N\setminus\{0\}$, there are at least  $ \frac{ \alpha N}{4C}$ values of $a\in\mathbb{Z}_N$ such that $A'\cap P_{a,d}$ contains a non--trivial three--term arithmetic progression in $\mathbb{Z}_N$. Each such progression  can be in at most $M^2$ sets of the form $P_{a,d}$. Thus we have that $Z\ge \frac{ \alpha N(N-1)}{4M^2C}$. Finally, we have that
	\[
	\Lambda(a_1,a_1,a_1)=\sum_{x,d\in\mathbb{Z}_N}a_1(x)a_1(x+d)a_1(x+2d)\ge Z\frac{\alpha^3}{N^3C^3}\ge \frac{\alpha^4}{4C^4M^2}\cdot\frac{N-1}{N}\cdot N^{-1}\ge
	\]
	
	\[
	\frac{\alpha^4}{8C^4M^2}\cdot N^{-1}\ge \frac{\alpha^4}{32C^4e^{2D\alpha^{-1}\log^5(1/\alpha)}}\cdot N^{-1} \ge C_2N^{-1}e^{-C_3\alpha^{-1}\log^5(1/\alpha)}
	\]
	for suitable positive constants $C_2,C_3$ as desired. For $N<M$, we trivially have 
	
	\[
	\Lambda(a_1,a_1,a_1)=\sum_{x,d\in\mathbb{Z}_N}a_1(x)a_1(x+d)a_1(x+2d)\ge Z\frac{\alpha^3}{N^3C^3}\ge \frac{\alpha^3}{NM^2C^3}
	\] 
	since $Z$ must contain the trivial arithmetic progressions and $A'\neq 
	\emptyset$ since $|A'|\ge \frac{\alpha N}{2C}$. We conclude with the same calculation as before.
\end{proof}
\begin{proof}[Concluding the Proof of Theorem 1.5]
	Remember that we have assumed for the sake of a contradiction that there exists a 3AP-free $A_0\subseteq \mathbb{P}_B$ with positive upper relative density. All the previous lemmas are applicable. Fix $r\in (2+\frac{62-62/c_2}{16/c_1+17/c_2-32},3)$, then from Lemmas~$\ref{UBound}$, $\ref{LBound}$, we have that there exist positive constants $C_1,C_2,C_3$ such that 
	
	\[
	C_2N^{-1}e^{-C_3\alpha^{-1}\log^5(\alpha^{-1})}\le\Lambda(a_1,a_1,a_1)\le C_1 N^{-3/2}+C_1N^{-1}(\varepsilon^2\delta^{-r}+\delta^{2-r/r'})
	\]
	We will choose $\varepsilon,\delta\in(0,1)$ such that the above inequality fails and such that $\varepsilon^k\ge \log\log(W)/W$, so that Lemma $\ref{L4HL3}$ will be applicable. More specifically, we will choose two  positive constants $C_4,C_5$ such that \[\delta=e^{-C_4\alpha^{-1}\log^5(1/\alpha)},\,\varepsilon=e^{-C_5\alpha^{-1}\log^5(1/\alpha)}\]
We have
	
	\[
	C_2N^{-1}e^{-C_3\alpha^{-1}\log^5(\alpha^{-1})}\le C_1 N^{-3/2}+C_1N^{-1}\Big(e^{(-2C_5+rC_4)\alpha^{-1}\log^5(1/\alpha)}+e^{-C_4(2-r/r')\alpha^{-1}\log^5(1/\alpha)}\Big)
	\]
	Thus
	\[
	e^{-C_3\alpha^{-1}\log^5(\alpha^{-1})}\Big(C_2-C_1e^{-(2C_5-rC_4-C_3)\alpha^{-1}\log^5(\alpha^{-1})}-C_1e^{-(C_4(2-r/r')-C_3)\alpha^{-1}\log^5(\alpha^{-1})} \Big)\le C_1N^{-1/2}
	\]
	We can choose a sufficiently large $C_4>0$ such that $C_1e^{-(C_4(2-r/r')-C_3)\alpha^{-1}\log^5(\alpha^{-1})}\le C_2/4$, and then choose a sufficiently large $C_5>0$ such that $C_1e^{-(2C_5-rC_4-C_3)\alpha^{-1}\log^5(\alpha^{-1})}\le C_2/4$. Then
	
	\[
	C_2/2e^{-C_3\alpha^{-1}\log^5(1/\alpha)}\le C_1N^{-1/2}
	\]
	We will have a contradiction provided that the assumptions of Lemma~$\ref{L4HL3}$ are in place. We note that $\varepsilon^k=\varepsilon^{|R|}\ge \varepsilon^{C\delta^{-r}}$, where $C$ is a positive constant guaranteed by the estimate in $\ref{lastarg}$. It suffices to show that $\varepsilon^{C\delta^{-r}}\ge \log\log W/W$ which is equivalent to 
	\[
	Ce^{rC_4\alpha^{-1}\log^5(1/\alpha)}C_5\alpha^{-1}\log^5(1/\alpha)\le \log\big(W/\log\log W\big)
	\]
	To show that this is the case for sufficiently large $N$, it suffices to show that 
	\[
	Ce^{rC_4\alpha^{-1}\log(1/\alpha)^5}C_5\alpha^{-1}\log^5(1/\alpha)\le \log\bigg(\frac{\frac{1}{8}\log\log N}{\log\log\big(\frac{1}{8}\log\log(N)\big)}\bigg) 
	\]
	which is true for $N$ large enough to make $\alpha\gtrsim \frac{(\log\log\log\log\log N)^6}{\log\log\log\log N}$. The proof of Theorem $\ref{T2}$ is complete.
\end{proof}

 \section{Proof of Main Lemma}\label{expest}
 This section is devoted to the proof of Lemma~\ref{PR2}, which is the main tool that allows us to use Bourgain–Green's result, see Theorem~$\ref{B-G}$, to obtain our Restriction Theorem, see Theorem~$\ref{RestOur}$. The methods used are analogous to the ones in Sections 6 and 7 of \cite{MMR} and the main difficulty here lies in the technical complications that the sophisticated nature of the sets $B$ bring. We fix $c_1\in[1,16/15)$ and $c_2\in[1,17/16)$, $h_1$, $h_2$, $\psi$ and $B$ as in the introduction and all the implied constants may depend on them. Let us mention that if $c_1=1$ then we fix $\sigma_1$ as in Lemma~2.14 in \cite{MMR}, otherwise let $\sigma_1$ be the constant function $1$. We use the basic properties of the functions, described in Lemma~2.6  and Lemma~2.14 in \cite{MMR} without further mention.  Before attempting to prove the main Lemma, we collect some useful intermediate results.
 
 \begin{lemma}\label{Lpsi}
 	Let $m\in\mathbb{Z}\setminus\{0\}$, $l\in\mathbb{N}$, $j\ge 0$, $X\ge 1$, $\alpha\in\mathbb{R}$ and $s\in\{0,1\}$. Then 
 	\[
 	\bigg|\sum_{k=1}^Xe^{2\pi i (\alpha j k l+m(\varphi_1(kl)-s\psi(kl)))}\bigg|\lesssim |m|^{1/2}\log(lX)lX(\sigma(lX)\varphi_1(lX))^{-1/2}
 	\]
 	and more precisely, for all positive real numbers $Y,Y'$ we have
 	\[
 	\bigg|\sum_{Y<k\le Y'\le 2Y}e^{2\pi i (\alpha j k l+m(\varphi_1(kl)-s\psi(kl)))}\bigg|\lesssim
 	|m|^{1/2}lY(\sigma_1(lY)\varphi_1(lY))^{-1/2}\]
 \end{lemma} 
 \begin{proof}
 The proof of the result for $s=0$ is given in \cite{MMR}, see Lemma~6.7, page 21. Let $Y\in[1,X]$ and $Y'\in[Y+1,2Y]$ and $F(t)=\alpha jlt+m(\varphi_1(lt)-\psi(lt))$. Then $F''(t)=ml^2(\varphi_1''(lt)-\psi''(lt))$.
 If $c_1>1$, then we have that for all $t\in[Y,2Y]$
 	\[
 	|ml^2\varphi_1''(lt)|= \bigg|ml^2 \frac{\varphi_1(lt)(\gamma_1+\theta^{(1)}_1(lt))(\gamma_1-1+\theta^{(1)}_2(lt))}{(lt)^2}\bigg|\simeq ml^2\frac{\varphi_1(lY)}{(lY)^2}
 	\]
where $\theta^{(j)}_i$ is the function $\theta_i$ appearing in Lemma~2.14, \cite{MMR} for $\varphi_j$. Also, 
 	\[
 	|ml^2\psi''(lt)|\lesssim
 	|ml^2\varphi_2'''(lt)|= \bigg|ml^2 \frac{\varphi_2(lt)(\gamma_2+\theta^{(2)}_1(lt))(\gamma_2-1+\theta^{(2)}_2(lt))(\gamma_2-2+\theta^{(2)}_3(lt))}{(lt)^3}\bigg|\lesssim ml^2\frac{\varphi_2(lY)}{(lY)^3}
 	\]
 	and thus
 	\[
 	|F''(t)|\lesssim ml^2\frac{\varphi_1(lY)}{(lY)^2}+ml^2\frac{\varphi_2(lY)}{(lY)^3}\lesssim  ml^2\frac{\varphi_1(lY)}{(lY)^2},\text{ since }\frac{\varphi_2(lY)}{\varphi_1(lY)lY}\lesssim 1
 	\]
Also,
 	\[
 	|F''(t)|\gtrsim ml^2\frac{\varphi_1(lY)}{(lY)^2}-ml^2\frac{\varphi_2(lY)}{(lY)^3} =ml^2\frac{\varphi_1(lY)}{(lY)^2}\bigg(1-\frac{\varphi_2(lY)}{\varphi_1(lY)lY} \bigg)\gtrsim ml^2\frac{\varphi_1(lY)}{(lY)^2} 
 	\text{, since }\lim_{Y\to\infty}\frac{\varphi_2(lY)}{\varphi_1(lY)lY}=0
\]
Thus $|F''(t)|\simeq ml^2\frac{\varphi_1(lY)}{(lY)^2}$. \\If $c_1=1$, then we have that for all $t\in[Y,2Y]$
 	\[
 	|ml^2\varphi_1''(lt)|= \bigg|ml^2 \frac{\varphi_1(lt)(\gamma_1+\theta^{(1)}_1(lt))\sigma_1(lt)\tau_1(lt)}{(lt)^2}\bigg|\simeq ml^2\frac{\varphi_1(lY)\sigma_1(lY)}{(lY)^2}
 	\]
 	and thus
 	\[
 	|F''(t)|\lesssim ml^2\frac{\varphi_1(lY)\sigma_1(lY)}{(lY)^2}+ml^2\frac{\varphi_2(lY)}{(lY)^3}= ml^2\frac{\varphi_1(lY)\sigma_1(lY)}{(lY)^2}\bigg(1+\frac{\varphi_2(lY)}{\varphi_1(lY)\sigma_1(lY)lY} \bigg)\lesssim  ml^2\frac{\varphi_1(lY)\sigma_1(lY)}{(lY)^2}
 	\]
 	since $\frac{\varphi_2(lY)}{\varphi_1(lY)\sigma_1(lY)lY}\lesssim 1$, because $\frac{(lY)^{\gamma_2-\gamma_1-1}\sigma_1(lY)^{-1}\ell_{\varphi_2}(lY)}{\ell_{\varphi_1}(lY)}\to 0$ as $Y\to\infty$. Finally,
 	\[
 	|F''(t)|\gtrsim ml^2\frac{\varphi_1(lY)\sigma_1(lY)}{(lY)^2}-ml^2\frac{\varphi_2(lY)}{(lY)^3}= ml^2\frac{\varphi_1(lY)\sigma_1(lY)}{(lY)^2}\bigg(1-\frac{\varphi_2(lY)}{\varphi_1(lY)\sigma_1(lY)lY} \bigg)\gtrsim  ml^2\frac{\varphi_1(lY)\sigma_1(lY)}{(lY)^2}
 	\]
Thus $|F''(t)|\simeq ml^2\frac{\varphi_1(lY)\sigma_1(lY)}{(lY)^2}$. In both cases, we may apply Van der Corput Lemma, see Corollary 8.13, page 208 in \cite{IWKO}, (assume $\sigma_1(x)=1$, whenever $c_1>1$).
 	
 	\[
 	\bigg|\sum_{Y<k\le Y'\le 2Y}e^{2\pi i (\alpha j k l+m(\varphi_1(kl)-\psi(kl)))}\bigg|\lesssim Y|m|^{1/2}\frac{l(\sigma_1(lY)\varphi_1(lY))^{1/2}}{lY}+|m|^{-1/2}\frac{l^{-1}(\sigma_1(lY)\varphi_1(lY))^{-1/2}}{(lY)^{-1}}\lesssim 
 	\]
 	\[
 	|m|^{1/2}lY\bigg(\frac{\sigma_1(lY)\varphi_1(lY)}{(lY)^2}\bigg)^{1/2}+|m|^{-1/2}Y(\sigma_1(lY)\varphi_1(lY))^{-1/2}\lesssim
 	\]
 	\[
 	|m|^{1/2}lY(\sigma_1(lY)\varphi_1(lY))^{-1/2}\bigg(\frac{\sigma_1(lY)\varphi_1(lY)}{(lY)^2}+1 \bigg)\lesssim |m|^{1/2}lY(\sigma_1(lY)\varphi_1(lY))^{-1/2}
 	\]
 since $\sigma_1(x)\lesssim 1$ and $\varphi_1(x)\lesssim x^2$. To conclude, let's estimate using the dyadic pieces
 	
 	\[
 	\bigg|\sum_{k=1}^Xe^{2\pi i (\alpha j k l+m(\varphi_1(kl)-\psi(kl)))}\bigg|\lesssim \log(X)\sup_{Y\in[1,X]}\bigg\{|m|^{1/2}lY(\sigma_1(lY)\varphi_1(lY))^{-1/2}\bigg\}
 	\]
 	\[
 	\lesssim |m|^{1/2}\log(lX)lX(\sigma(lX)\varphi_1(lX))^{-1/2}
 	\]
 	since $x(\sigma(x)\varphi_1(x))^{-1/2}=x\varphi_1(x)^{-1/2}\sigma_1^{-1/2}(x)$ is increasing.
 \end{proof}

 Finally, we will need the following Lemma. Let us denote by $\Lambda$ the von
Mangoldt’s function as usual
 \[\Lambda(n)=\left\{
\begin{array}{ll}
      \log(n), & \text{if }n=p^k\text{ for some }p\in\mathbb{P}\text{ and }k\in\mathbb{N}, \\
      0,&\text{ otherwise.} \\
\end{array} 
\right. 
\]
and let us define $\Lambda_{a,q}(n)=\Lambda(n)1_{P_{a,q}}(n)$, where $P_{a,q}=\{\,n\in\mathbb{N}:n\equiv a\Mod q\,\}$.
 \begin{lemma}\label{LLambda}
 	Let $P\in\mathbb{N}$, $\xi\in\mathbb{T}$ and $M=P^{1+\chi+\varepsilon-(99/100)\gamma_2}$ where $\chi>0$, $0<\varepsilon<\chi/100$ are such that such that $16(1-\gamma_1)+17(1-\gamma_2)+31\chi\le1$. If we let $a,q\in\mathbb{N}$ such that $0\le a<q$ and $(a,q)=1$, then for every $m\in \mathbb{Z}$ such that $0<|m|\le M$ and every $P_1\in\mathbb{N}$ we have
 	
 	\[
 	\bigg|\sum_{P<k\le P_1\le 2P} \Lambda_{a,q}(k)e^{2\pi i (k\xi-m\varphi_1(k))}\bigg|\lesssim
 	\] 
 	\[
 	|m|^{1/2}\log(P_1)P_1^{4/3}\big(\sigma_1(P_1)\varphi_1(P_1)\big)^{-1/2}+|m|^{1/6}\log^6(P_1)P_1^{13/12}(\sigma_1(P_1)\varphi_1(P_1))^{-1/6}
 	\]
 \end{lemma}
 \begin{proof}
 	Recall that
 	\[
 	1_{P_{a,q}}(k)=\frac{1}{q}\sum_{s=0}^{q-1}e(s(k-a)/q)\text{, where }e(x)=e^{2 \pi i x}
 	\]
 	we may write
 	\[
 	\sum_{P<k\le P_1\le 2P} \Lambda_{a,q}(k)e^{2\pi i (k\xi-m\varphi_1(k))}=\sum_{P<k\le P_1\le 2P} \Lambda(k)\frac{1}{q}\sum_{s=0}^{q-1}e^{2\pi i(s(k-a)/q)}e^{2\pi i (k\xi-m\varphi_1(k))}=
 	\]
 	\[
 	\frac{1}{q}\sum_{s=0}^{q-1}e^{-2\pi isa/q}\sum_{P<k\le P_1\le 2P} \Lambda(k)e^{2\pi i (k(\xi+s/q)-m\varphi_1(k))}
 	\]
 	and therefore it suffices to show that 
 	\[
 	\sum_{P<k\le P_1\le 2P} \Lambda(k)e^{2\pi i (k\alpha-m\varphi_1(k))}\lesssim |m|^{1/2}\log(P_1)P_1^{4/3}\big(\sigma_1(P_1)\varphi_1(P_1)\big)^{-1/2}+
 	\]
 	\[|m|^{1/6}\log^6(P_1)P_1^{13/12}\big(\sigma_1(P_1)\varphi_1(P_1)\big)^{-1/6} \]
 	where the implied constant is uniform in $\alpha=\xi+s/q$ where $\xi\in\mathbb{T}$, $0\le s\le q-1$. To that end, we use Vaughan's identity, which we state here for the sake of clarity.
 	\begin{lemma}Let $v,w\ge 1$ be real numbers and let $n\in\mathbb{N}$ be such that $n>v$, then
 		\[
 		\Lambda(n)=\sum_{\substack{b|n\\b\le w}}
 		\mu(b)\log(n/b)-\mathop{\sum\sum}\limits_{\substack{bc|n\\b\le w,c\le v}}\mu(b)\Lambda(c)+\mathop{\sum\sum}\limits_{\substack{bc|n\\b>w,c>v}}\mu(b)\Lambda(c)
 		\] 
 		Or equivalently, for every $n>v$ we have
 		\[
 		\Lambda(n)=\sum_{kl=n\text{, }l\le w}\log(k)\mu(l)-\sum_{l\le vw}\sum_{kl=n}\Pi_{v,w}(l)+\sum_{kl=n\text{, }k>v\text{, }l> w}\Lambda(k)\Xi_w(l)
 		\]
 		where \[\Pi_{v,w}(l)=\sum_{\substack{rs=l\\r\le v\text{, }s\le w}}\Lambda(r)\mu(s)
 		\]
 		and
 		\[
 		\Xi_w(l)=\sum_{\substack{d|l\\d>w}}\mu(d)
 		\]
 	\end{lemma}
 	\begin{proof}See Proposition 13.4, page 345 in \cite{IWKO}.
 	\end{proof}
 We use this result for $w=v$ so let $\Pi_{v,v}=\Pi_v$ for simplicity. More specifically, set $v=w=P_1^{1/3}$ where $P<P_1\le 2P$. For sufficiently large $P$, $n\in(P,P_1]$ is such that $n>v$ and thus Vaughan's identity is applicable. We have
 	\[
 	\sum_{P<n\le P_1\le 2P} \Lambda(n)e^{2\pi i (n\alpha-m\varphi_1(n))}=\sum_{P<n\le P_1\le 2P}\sum_{kl=n\text{, }l\le v}\log(k)\mu(l)e^{2\pi i (n\alpha-m\varphi_1(n))}\]
 	\[
 	-\sum_{P<n\le P_1\le 2P}\sum_{l\le v^2}\sum_{kl=n}\Pi_{v}(l)e^{2\pi i (n\alpha-m\varphi_1(n))}\]
 	\[
 	+\sum_{P<n\le P_1\le 2P}\sum_{kl=n\text{, }k>v\text{, }l> v}\Lambda(k)\Xi_v(l)e^{2\pi i (n\alpha-m\varphi_1(n))}=
 	\] 
 	\[
 	\sum_{l\le v} \sum_{P/l<k\le P_1/l}\log(k)\mu(l)e^{2\pi i(\alpha k l-m\varphi_1(kl))}
 	\]
 	\[
 	-\sum_{l\le v}\sum_{P/l<k\le P_1/l}\Pi_v(l)e^{2\pi i (\alpha kl-m\varphi_1(kl))}-\sum_{v< l\le v^2}\sum_{P/l<k\le P_1/l}\Pi_v(l)e^{2\pi i (\alpha kl-m\varphi_1(kl))}
 	\]
 	\[
 	+\sum_{v<l\le P_1/v}\sum_{\substack{P/l<k\le P_1/l\\k>v}}\Lambda(k)\Xi_v(l)e^{2\pi i (\alpha kl-m\varphi_1(kl))}=S_1-S_{2,1}-S_{2,2}+S_3
 	\]
 	where we changed the order of summation, and we have named the four terms appearing in the final sum by $S_1$, $-S_{2,1}$, $-S_{2,2}$ and $S_3$ respectively. The proof is now reduced to estimating these four terms. For $S_1$, we use summation by parts (for the specific version we are using see Theorem~A4, page 304 in \cite{NM}). Let's denote by $U_l(t)=\sum_{P/l<k\le t}e^{2\pi i(\alpha k l-m\varphi_1(kl))}$, then
 	\[
 	|S_1|\le \sum_{l\le v} |\mu(l)| \bigg|\sum_{P/l<k\le P_1/l}\log(k)e^{2\pi i(\alpha k l-m\varphi_1(kl))}\bigg|=\sum_{l\le v} |\mu(l)| \bigg|U_l(P_1/l)\log(P_1/l)-\int_{P/l}^{P_1/l}U_l(t)/tdt\bigg|\le
 	\]
 	\[
 	\sum_{l\le v}|U_l(P_1/l)|\log(P_1/l)+\sup_{P/l<t\le P_1/l}|U_l(t)|\big(\log(P_1/l)-\log(P/l)\big)\le 
 	\]
 	\[
 	2\log(P_1)\sum_{l\le v}\sup_{P/l<t\le P_1/l}|U_l(t)|
 	\] For every $t\in(P/l,P_1/l]$, $P_1/l\le 2P/l$, we estimate the dyadic pieces of the form 
 	\[\bigg|\sum_{Y<k\le Y'\le 2Y}e^{2\pi i (\alpha j  k l+m'(\varphi_1(kl)))}\bigg|\lesssim |m'|^{1/2}lY\big(\sigma_1(lY)\varphi_1(lY)\big)^{-1/2}
 	\]
 by applying Lemma~$\ref{Lpsi}$ for $Y=P/l$, $m'=-m$, $j=1$, $s=0$, and $Y'=x$, to obtain
 	\[
 	|U_l(x)|=\bigg|\sum_{P/l<k\le x}e^{2\pi i(\alpha k l-m\varphi_1(kl))}\bigg|\lesssim |m|^{1/2}P\big(\sigma_1(P)\varphi_1(P)\big)^{-1/2}\lesssim |m|^{1/2}P_1\big(\sigma_1(P_1)\varphi_1(P_1)\big)^{-1/2}  
 	\] where we have used the fact that $x(\sigma(x)\varphi_1(x))^{-1/2}$ is increasing. Thus we get
 	\[
 	|S_1|\lesssim  \log(P_1)\sum_{l\le v}\sup_{P/l<t\le P_1/l}|U_l(t)|\lesssim \log(P_1)v|m|^{1/2}P_1\big(\sigma_1(P_1)\varphi_1(P_1)\big)^{-1/2}=
 	\]
 	\[ |m|^{1/2}\log(P_1)P_1^{4/3}\big(\sigma_1(P_1)\varphi_1(P_1)\big)^{-1/2}
 	\]
 	For $S_{2,1}$, the estimates follow from similar considerations. Firstly, notice that
 	\[
 	|S_{2,1}|=\bigg|\sum_{l\le v}\sum_{P/l<k\le P_1/l}\Pi_v(l)e^{2\pi i (\alpha kl-m\varphi_1(kl))}\bigg|\le \sum_{l\le v}|\Pi_v(l)||U_l(P_1/l)| 
 	\] and also that 
 	\[
 	|\Pi_v(l)|=\bigg|\sum_{\substack{r,s\in\mathbb{N},\\rs=l\\r\le v\text{, }s\le v}}\Lambda(r)\mu(s)\bigg|\le \sum_{r|l}\Lambda(r)=\log(l)\le \log(P_1)\text{ for }l\le v=P_1^{1/3}
 \]
 Thus we get 
 \[
 |S_{2,1}|\le\sum_{l\le v}|\Pi_v(l)||U_l(P_1/l)|\le  
 \log(P_1)\sum_{l\le v}\sup_{P/l<t\le P_1/l}|U_l(t)|
 \]
 and we can conclude exactly as in the case of $S_1$.
 
 We now focus on $S_{2,2}$, and $S_3$ which will be treated simultaneously. We will use the dyadic pieces of the sums
 	\[
 	|S_{2,2}|=\bigg| \sum_{v< l\le v^2}\sum_{P/l<k\le P_1/l}\Pi_v(l)e^{2\pi i (\alpha kl-m\varphi_1(kl))}\bigg|\lesssim
 	\]
 	\[
 	\log(v^2)\log(P_1/v)\sup_{L\in[v,v^2]}\sup_{K\in[P/v^2,P_1/v]}\sup_{L'\in(L,2L]}\sup_{K'\in(K,2K]}\bigg|
 	\sum_{L< l\le L'\le2L}\sum_{\substack{K<k\le K'\le2K\\P<kl\le P_1}}\Pi_v(l)e^{2\pi i (\alpha kl-m\varphi_1(kl))}\bigg|\lesssim
 	\]
 	\begin{equation}\label{DyadPc1}
 		\log^2(P_1)\sup_{L\in[v,v^2]}\sup_{K\in[P/v^2,P_1/v]}\sup_{L'\in(L,2L]}\sup_{K'\in(K,2K]}\bigg|
 		\sum_{L< l\le L'\le2L}\sum_{\substack{K<k\le K'\le2K\\P<kl\le P_1}}\Pi_v(l)e^{2\pi i (\alpha kl-m\varphi_1(kl))}\bigg|
 	\end{equation}
 For $|S_3|$, we have
 	\[
 	|S_3|=\bigg|\sum_{v<l\le P_1/v}\sum_{\substack{P/l<k\le P_1/l\\k>v}}\Lambda(k)\Xi_v(l)e^{2\pi i (\alpha kl-m\varphi_1(kl))}\bigg|\lesssim
 	\]
 	\[
 	\log^2(P_1/v)\sup_{L\in[v,P_1/v]}\sup_{K\in[v,P_1/v]}\sup_{L'\in(L,2L]}\sup_{K'\in(K',2K]}\bigg|\sum_{L<l\le L'\le 2L}\sum_{\substack{K<k\le K'\le 2K\\P<kl\le P_1}}\Lambda(k)\Xi_v(l)e^{2\pi i (\alpha kl-m\varphi_1(kl))}\bigg|\lesssim
 	\]
 	\begin{equation}\label{DyadPc2}
 		\log^2(P_1)\sup_{L\in[v,P_1/v]}\sup_{K\in[v,P_1/v]}\sup_{L'\in(L,2L]}\sup_{K'\in(K',2K]}\bigg|\sum_{L<l\le L'\le 2L}\sum_{\substack{K<k\le K'\le 2K\\P<kl\le P_1}}\Lambda(k)\Xi_v(l)e^{2\pi i (\alpha kl-m\varphi_1(kl))}\bigg|
 	\end{equation}

 	On the one hand, we have
 	\[
 	\sum_{L<l\le L'\le 2L}|\Pi_v(l)|^2\le\sum_{L<l\le 2L}\log^2(l) \lesssim \log^2(L)L
 	\]
 	On the other hand, if we let $d(n)=|\{\,d\in\mathbb{N}\,:\,d|n\,\}|$, we have
 	\[
 	\sum_{L< l\le L'\le2L}|\Xi_v(l)|^2\lesssim \sum_{1\le l\le 2L} d(l)^2\lesssim L\log^3(L)
 	\text{, (see Theorem A.14, page 313 in \cite{NM})}\]
We now use the following technical Lemma.
 	\begin{lemma}\label{6.12'}
 		Let $L,K\in\mathbb{N}$
 		and  $m\in\mathbb{Z}\setminus\{0\}$. If $|m|\min\{L,K\}\le\varphi_1(LK)\sigma_1(LK)$ and $\varphi_1(LK)\le\min\{L,K\}^4$, then
 		\[
 		\bigg|\sum_{L<l\le L'\le 2L}\sum_{\substack{K<k\le K'\le 2K\\P<kl\le P_1}}\Delta_1(l)\Delta_2(k)e^{2\pi i(\alpha kl-m\varphi_1(kl))}\bigg|\lesssim\]
 		\[
 		|m|^{1/6}\log^2(L)\log^2(K)(\sigma_1(LK)\varphi_1(LK))^{-1/6}\min\{L,K\}^{1/6}KL
 		\] for every sequence of complex numbers $\big(\Delta_1(l)\big)_{l\in(L,2L]}$ and $\big(\Delta_2(k)\big)_{k\in(K,2K]}$ having the property that 
 		\[
 		\sum_{L<l\le 2L}|\Delta_1(l)|^2\lesssim L\log^3(L)\text{ and }\sum_{K<k\le 2K}|\Delta_2(k)|^2\lesssim K\log^3(K)
 		\] 
 	\end{lemma}
 	\begin{proof}
 		For the proof of this result we refer to \cite{MMR}, Lemma~6.12, page 23.
 	\end{proof}
 	We wish to use the above result to estimate the dyadic pieces in $\ref{DyadPc1}$ and $\ref{DyadPc2}$. For any $K,L\in\mathbb{N}$ which make the dyadic piece nonempty, there exist natural numbers $k,l$ such that
 	\[
 	KL< kl \le P_1\text{ and }KL\ge \frac{kl}{4}>P/4\ge P_1/8
 	\]
 	and thus $P_1/8\le KL\le P_1$. For $S_{2,2}$, notice that
 	\[
 	K\le P_1/v=P_1^{2/3}\text{ and }K>P/v^2\ge P_1^{1/3}/2\text{, and thus }K\in\big[P_1^{1/3}/2,P_1^{2/3}\big] 
 	\]
 	and similarly
 	\[
 	L\in[v,v^2]=\big[P_1^{1/3},P_1^{2/3}\big]\subseteq \big[P_1^{1/3}/2,P_1^{2/3}\big] 
 	\]
 	For $S_3$, notice that
 	\[
 	K,L\in[v,P_1/v]=\big[P_1^{1/3},P_1^{2/3}\big]\subseteq \big[P_1^{1/3}/2,P_1^{2/3}\big] 
 	\]Therefore, in either case, $K,L\in \big[P_1^{1/3}/2,P_1^{2/3}\big] $. Also, since $KL\le P_1$, we must have that $\min\{L,K\}\le P_1^{1/2}$ and $\varphi_1(LK)\le \varphi_1(P_1)\le P_1\le \min\{L,K\}^4$, since $\min\{L,K\}^4> P_1^{4/3}/16\ge P_1$ for sufficiently large $P_1$. Finally, we have
 	\[
 	|m|\min\{L,K\}\le M P_1^{1/2}=P^{1+\chi+\varepsilon-99\gamma_2/100}P_1^{1/2}\le P_1^{3/2+\chi+\varepsilon-99\gamma_2/100}
 	\] We claim that $3/2+\chi+\varepsilon-99\gamma_2/100<\gamma_1$. To show this, considering that $0<\varepsilon\le \chi/100$, it suffices to show 
 	\[
 	3/2+101\chi/100-99\gamma_2/100-\gamma_1<0
 	\]
 	which, in turn, is equivalent to 
 	
 	\[
 	101\chi/49+99/49(1-\gamma_2)+100/49(1-\gamma_1)<1
 	\]but this is true since
 	\[
 	101\chi/49+99/49(1-\gamma_2)+100/49(1-\gamma_1)<31\chi+17(1-\gamma_2)+16(1-\gamma_1)\le1
 	\]
 	by our assumptions. The proof of the claim is complete. Now if we let $\delta=-(3/2+\chi+\varepsilon-99\gamma_2/100-\gamma_1)>0$,
 	then
 	\[
 	|m|\min\{L,K\}\le P_1^{\gamma_1-\delta}\lesssim \varphi_1(P_1)\sigma_1(P_1)\lesssim \varphi_1(LK)\sigma_1(LK)
 	\]Thus we may use the Lemma $\ref{6.12'}$ for appropriate $\Delta_1,\Delta_2$ depending on whether we deal with $S_{2,2}$ or $S_3$ to obtain	
 	\[
 	\bigg|\sum_{L<l\le L'\le 2L}\sum_{\substack{K<k\le K'\le 2K\\P<kl\le P_1}}\Delta_1(l)\Delta_2(k)e^{2\pi i(\alpha kl-m\varphi_1(kl))}\bigg|\lesssim
 	\]
 	
 	\[
 	|m|^{1/6}\log^2(L)\log^2(K)(\sigma_1(LK)\varphi_1(LK))^{-1/6}\min\{L,K\}^{1/6}KL\lesssim 
 	\]
 	\[
 	|m|^{1/6}\log^4(P_1)(\sigma_1(P_1)\varphi_1(P_1))^{-1/6}P_1^{1/12}P_1=|m|^{1/6}\log^4(P_1)(\sigma_1(P_1)\varphi_1(P_1))^{-1/6}P_1^{13/12}
 	\]
 	And thus
 	\[
 	|S_{2,2}|,|S_3|\lesssim |m|^{1/6}\log^6(P_1)(\sigma_1(P_1)\varphi_1(P_1))^{-1/6}P_1^{13/12}
 	\]
 	This concludes the proof of Lemma $\ref{LLambda}$.
 \end{proof}
 We are now ready to prove the Main Lemma.
 \begin{proof}[Proof of Lemma~$\ref{PR2}$]Let $a,q\in \mathbb{Z}$ be such that $0\le a\le q-1$ and $(a,q)=1$, and let $\chi>0$ be such that $16(1-\gamma_1)+17(1-\gamma_2)+31\chi\le1$.
 	By Lemma 2.2 in \cite{HLMP}, we have that $\lfloor\varphi_1(n)\rfloor-\lfloor\varphi_1(n)-\psi(n)\rfloor=1_B(n)$, and thus
 	\[
 	\sum_{\substack{p\in\mathbb{P}_B\cap [N]\\p\equiv a\Mod q}}\psi(p)^{-1}\log(p)e(p\xi)=\sum_{\substack{p\in\mathbb{P}\cap [N]\\p\equiv a\Mod q}}\psi(p)^{-1}\log(p)(\lfloor\varphi_1(p)\rfloor-\lfloor\varphi_1(p)-\psi(p)\rfloor)e(p\xi)=
 	\]
 	
 	\[
 	\sum_{\substack{p\in\mathbb{P}\cap [N]\\p\equiv a\Mod q}}\psi(p)^{-1}\log(p)\big(\psi(p)+(\Phi(\varphi_1(p)-\psi(p)))-\Phi(\varphi_1(p))\big)e(p\xi)=
 	\]
 	(where $\Phi(x)=\{x\}-1/2=x-\lfloor x\rfloor-1/2 $)
 	\[
 	\sum_{\substack{p\in\mathbb{P}\cap [N]\\p\equiv a\Mod q}}\log(p)e(p\xi)+ \sum_{\substack{p\in\mathbb{P}\cap [N]\\p\equiv a\Mod q}}\psi(p)^{-1}\log(p)(\Phi(\varphi_1(p)-\psi(p)))-\Phi(\varphi_1(p))e(p\xi)=
 	\]
 	
 	\[
 	\sum_{\substack{p\in\mathbb{P}\cap [N]\\p\equiv a\Mod q}}\log(p)e(p\xi)+ \sum_{n\in [N]}\psi(n)^{-1}\Lambda_{a,q}(n)\big(\Phi(\varphi_1(n)-\psi(n)))-\Phi(\varphi_1(n))\big)e(n\xi)-
 	\]
 	\[
 	-\bigg(
 	\sum_{\substack{n\in[N]\\n=p^s\text{ for some }p\in\mathbb{P},s\ge2}}\psi(n)^{-1}\Lambda_{a,q}(n)\big(\Phi(\varphi_1(n)-\psi(n))-\Phi(\varphi_1(n))\big)e(n\xi)\bigg)
 	\]
 The absolute value of the third term can be bounded by
 	\[
 	 \sum_{\substack{1\le p^s\le N,\\p\in\mathbb{P},s\ge2}}\psi(p^s)^{-1}\log(p)\lesssim \frac{N}{\varphi_2(N)}\sum_{\substack{1\le p^s\le N,\\p\in\mathbb{P},s\ge2}}\log(p)\text{, since }\psi(p^s)^{-1}\lesssim \varphi_2'(p^s)^{-1}\lesssim \varphi_2'(N)^{-1}\lesssim \frac{N}{\varphi_2(N)}
 	\]
 Notice that each prime $p$ will contribute $\log(p)$ to the sum exactly $s_p-1$ times where $s_p$ is the integer with the property $p^{s_p}\le N< p^{s_p+1}$ or equivalently $s_p=\lfloor \log(N)/\log(p)\rfloor$. Thus
 	\[
 	\sum_{\substack{1\le p^s\le N,\\p\in\mathbb{P},s\ge2}}\log(p)\le \sum_{\substack{1\le p^2\le N,\\p\in\mathbb{P}}}\bigg\lfloor \frac{\log(N)}{\log(p)}\bigg\rfloor \log(p) \le \log(N)\sum_{\substack{p\in\mathbb{P}\\p\le\sqrt{N}}}1\lesssim \frac{\log(N)\sqrt{N}}{\log(\sqrt{N})}\lesssim N^{1/2}
 	\]
 	where we have used the fact that $|\mathbb{P}\cap [1,x]|\lesssim x(\log x)^{-1}$. For every $\varepsilon'>0$, there exists a positive constant $C_{\varepsilon'}$ such that
 	\[
 	\bigg(\,\text{third term}\,\bigg)\lesssim \frac{N}{\varphi_2(N)}N^{1/2}\le C_{\varepsilon}N^{3/2-\gamma_2+\varepsilon'} 
 	\]
 For the choice $\varepsilon'=\gamma_2-3/2\chi-1/2$, we can verify that $\varepsilon'\ge 16/17-3/62-1/2>0$ and that $3/2-\gamma_2+\varepsilon'=1-3/2\chi$, and thus we have shown 
 	\begin{equation}\label{1red}
 		\begin{aligned}
 			&\sum_{\substack{p\in\mathbb{P}\cap B\cap [N]\\p\equiv a\Mod q}}\psi(p)^{-1}\log(p)e(p\xi)= \sum_{\substack{p\in\mathbb{P}\cap [N]\\p\equiv a\Mod q}}\log(p)e(p\xi)+\\
 			&+\sum_{n=1}^N\psi(n)^{-1}\Lambda_{a,q}(n)\big(\Phi(\varphi_1(n)-\psi(n))-\Phi(\varphi_1(n))\big)e(n\xi)+O(N^{1-3/2\chi})
 		\end{aligned}
 	\end{equation}
 	This concludes our first reduction. We now bound the second term in $\ref{1red}$ by looking at its dyadic pieces. To achieve this, we will use the estimates for exponential sums we have proven in the section, together with the Fourier Expansion of the function $\Phi$. More specifically, for $M\ge 1$, we know that 
 \[
 \Phi(x)=\sum_{0<|m|\le M}\frac{1}{2\pi im}e^{-2\pi imx}+g_M(x)
 \]
 with $g_M(x)=O\big(\min\big\{1,\frac{1}{M||x||}\big\}\big)$ and $\min\big\{1,\frac{1}{M||x||}\big\}=\sum_{m\in\mathbb{Z}}b_me^{2\pi i m x}$ where $|b_m|\lesssim \min\big\{\frac{\log(M)}{M},\frac{1}{|m|},\frac{M}{|m|^2}\big\}$, see section 2 in \cite{DRHB}. Let $P\in\mathbb{N}$, $P'\in[P+1,2P]$ and $M\ge 1$, then
 	
 	\[
 	\sum_{P<k\le P'\le 2P}\psi(k)^{-1}\Lambda_{a,q}(k)\big(\Phi(\varphi_1(k)-\psi(k))-\Phi(\varphi_1(k))\big)e(k\xi)=
 	\] 	
 \begin{equation}\label{secondred}
 		\begin{aligned}
 			&\sum_{0<|m|\le M}\frac{1}{2\pi im} \sum_{P<k\le P'\le 2P}\psi(k)^{-1}\Lambda_{a,q}(k)\big(e^{2\pi i(-m\varphi_1(k)+m\psi(k)+k\xi) }-e^{2\pi i (-m\varphi_1(k))+k\xi}\big)+\\
 			&+\sum_{P<k\le P'\le 2P}\psi(k)^{-1}\Lambda_{a,q}(k)\big(g_M(\varphi_1(k)-\psi(k))-g_M(\varphi_1(k))\big)e(k\xi)
 		\end{aligned}
 	\end{equation}
 	We estimate the second term of the above sum using Lemma~$\ref{Lpsi}$. We have
 	\[
 	\bigg|\sum_{P<k\le P'\le 2P}\psi(k)^{-1}\Lambda_{a,q}(k)\big(g_M(\varphi_1(k)-\psi(k))-g_M(\varphi_1(k))\big)e(k\xi)\bigg|\lesssim 
 	\]
 	\[
 	\sum_{P<k\le P'\le 2P}\psi(k)^{-1}\Lambda_{a,q}(k)\bigg(\min\bigg\{1,\frac{1}{M||\varphi_1(k)-\psi(k)||}\bigg\}+\min\bigg\{1,\frac{1}{M||\varphi_1(k)||}\bigg\}\bigg)
 	\]
 	Note that
 	\[
 	\sum_{P<k\le P'\le 2P}\psi(k)^{-1}\Lambda_{a,q}(k)\min\bigg\{1,\frac{1}{M||\varphi_1(k)-\psi(k)||}\bigg\}\lesssim \frac{\log(P)}{\varphi_2'(P)}\sum_{P<k\le P'\le 2P}\sum_{m\in\mathbb{Z}}b_me^{2\pi i m (\varphi_1(k)-\psi(k))}
 	\]
 	Using the estimate $|b_m|\lesssim M/|m|^2$, we conclude that the function is summable and by Fubini-Tonelli we have
 \[
 \bigg| \sum_{P<k\le P'\le 2P}\sum_{m\in\mathbb{Z}}b_me^{2\pi i m (\varphi_1(k)-\psi(k))}\bigg|\le \sum_{m\in\mathbb{Z}}|b_m|\bigg|\sum_{P<k\le P'\le 2P}e^{2\pi i m (\varphi_1(k)-\psi(k))}\bigg|=
 	\]
 	\[
 	|b_0|P+\sum_{0<|m|\le M}|b_m|\bigg|\sum_{P<k\le P'\le 2P}e^{2\pi i m (\varphi_1(k)-\psi(k))}\bigg|+\sum_{|m|> M}|b_m|\bigg|\sum_{P<k\le P'\le 2P}e^{2\pi i m (\varphi_1(k)-\psi(k))}\bigg|\lesssim
 	\]
 	\[\frac{\log(M)P}{M}+\sum_{0<|m|\le M}\frac{\log(M)}{M}|m|^{1/2}P(\sigma_1(P)\varphi_1(P))^{-1/2}+\sum_{|m|> M}\frac{M}{|m|^2}|m|^{1/2}P(\sigma_1(P)\varphi_1(P))^{-1/2}\lesssim
 	\]
 	\[
 	\frac{\log(M)P}{M}+\log(M)M^{1/2}P(\sigma_1(P)\varphi_1(P))^{-1/2}+M^{1/2}P(\sigma_1(P)\varphi_1(P))^{-1/2}\lesssim
 	\]
 	\[
 	 \frac{\log(M)P}{M}+\log(M)M^{1/2}P(\sigma_1(P)\varphi_1(P))^{-1/2}
 	\]
 	where the estimates are justified by Lemma~\ref{Lpsi} for $j=0$, $l=1$ and $s=1$, together with the estimates for $|b_m|$. Thus
 	\[
 	\sum_{P<k\le P'\le 2P}\psi(k)^{-1}\Lambda_{a,q}(k)\min\bigg\{1,\frac{1}{M||\varphi_1(k)-\psi(k)||}\bigg\}\lesssim 
 	\]
 	\[
 	\frac{\log(P)}{\varphi_2'(P)}\Big( \frac{\log(M)P}{M}+\log(M)M^{1/2}P(\sigma_1(P)\varphi_1(P))^{-1/2}\Big)
 	\]
 	With similar considerations (and by applying Lemma~$\ref{Lpsi}$ for $s=0$), one obtains
 	\[
 	\sum_{P<k\le P'\le 2P}\psi(k)^{-1}\Lambda_{b,m}(k)\min\bigg\{1,\frac{1}{M||\varphi_1(k)||}\bigg\}\lesssim \frac{\log(P)}{\varphi_2'(P)}\bigg( \frac{\log(M)P}{M}+\log(M)M^{1/2}P(\sigma_1(P)\varphi_1(P))^{-1/2}\bigg)
 	\]
 	and thus
 	\[
 	\sum_{P<k\le P'\le 2P}\psi(k)^{-1}\Lambda_{a,q}(k)\bigg(\min\bigg\{1,\frac{1}{M||\varphi_1(k)-\psi(k)||}\bigg\}+\min\bigg\{1,\frac{1}{M||\varphi_1(k)||}\bigg\}\bigg)\lesssim
 	\]
 	\[
 	\frac{\log(P)}{\varphi_2'(P)}\bigg( \frac{\log(M)P}{M}+\log(M)M^{1/2}P(\sigma_1(P)\varphi_1(P))^{-1/2}\bigg)
 	\]
 	Let's fix a number $\varepsilon$ such that $0<\varepsilon<\chi/100$ and let $M=P^{1-99\gamma_2/100+\chi+\varepsilon}$. Then for all $\varepsilon_2>0$, we have that
 	\[
 	\frac{\log(P)}{\varphi_2'(P)}\bigg( \frac{\log(M)P}{M}+\log(M)M^{1/2}P(\sigma_1(P)\varphi_1(P))^{-1/2}\bigg)\lesssim
 	\]
 	\[
 	\log^2(P)\frac{P^{99\gamma_2/100-\chi-\varepsilon}}{\varphi_2'(P)}+\log^2(P)\frac{P^{3/2-99\gamma_2/200+\chi/2+\varepsilon/2}}{\varphi_2'(P)(\sigma_1(P)\varphi_1(P))^{1/2}}\lesssim_{\varepsilon_2}
 	\]	
 	\[
 	\log^2(P)\bigg(P^{99\gamma_2/100-\chi-\varepsilon-\gamma_2+1+\varepsilon_2}+P^{5/2-\gamma_2-\gamma_1/2-99\gamma_2/200+\chi/2+\varepsilon/2+3\varepsilon_2/2}\bigg)
 	\]
 	since $\varphi_2'(x)\gtrsim x^{\gamma_2-1-\varepsilon_2/2}$ and $\sigma_1^{-1}(x)\lesssim_{\delta} x^{\delta}$ for all $\delta>0$. It suffices to show that there exists a positive number $\varepsilon_2$ such that
 	\[
 	-\gamma_2/100-\chi-\varepsilon+1+\varepsilon_2<1-\chi-\varepsilon
 	\]
 	and
 	\[
 	5/2-299\gamma_2/200-\gamma_1/2+\chi/2+\varepsilon/2+3\varepsilon_2/2<1-\chi-\varepsilon
 	\]
 	in order to conclude that
 	\begin{equation}\label{secondredhelp}
 		\begin{aligned}
 			&\bigg|\sum_{P<k\le P'\le 2P}\psi(k)^{-1}\Lambda_{a,q}(k)\big(g_M(\varphi_1(k)-\psi(k))-g_M(\varphi_1(k))\big)e(k\xi)\bigg|=O(P^{1-\chi-\varepsilon})
 		\end{aligned}
 	\end{equation}
 	The first inequality is equivalent to $\varepsilon_2<\gamma_2/100$. Remembering that $\varepsilon<\chi/100$, for the second inequality to be true, it suffices to have
 	\[
 	3/2-299\gamma_2/200-\gamma_1/2+303\chi/200+3/2\varepsilon_2<0
 	\]
 	or equivalently
 	
 	\[
 	\varepsilon_2<99/300\bigg(1-299/99(1-\gamma_2)-100/99(1-\gamma_1)-303\chi/99\bigg)
 	\]
 	We have that 
 	\[
 	299/99(1-\gamma_2)+100/99(1-\gamma_1)+303\chi/99 < 17(1-\gamma_2)+16(1-\gamma_1)+31\chi\le 1
 	\]
 	Thus, we may choose 
 	$
 	\varepsilon_2=\min\bigg\{\gamma_2/200,99/600\bigg(1-299/99(1-\gamma_2)-100/99(1-\gamma_1)-303\chi/99\bigg)\bigg\}>0
 	$
 	and both inequalities are satisfied. Now notice that from \ref{1red} together with \ref{secondred} and \ref{secondredhelp}, we get
 	\[
 	\bigg|\sum_{\substack{p\in\mathbb{P}\cap B\cap [N]\\p\equiv a\Mod q}}\psi(p)^{-1}\log(p)e(p\xi)-\sum_{\substack{p\in\mathbb{P}\cap [N]\\p\equiv a\Mod q}}\log(p)e(p\xi)\bigg|=
 	\]	
 	\[\bigg|\sum_{n=1}^N\psi(n)^{-1}\Lambda_{a,q}(n)\big(\Phi(\varphi_1(n)-\psi(n))-\Phi(\varphi_1(n))\big)e(n\xi)\bigg|+O(N^{1-3/2\chi})\lesssim
 	\]
 	\[
 	\log(N)\sup_{1\le P\le N}\bigg|\sum_{P<n\le P'\le 2P}\psi(n)^{-1}\Lambda_{a,q}(n)\big(\Phi(\varphi_1(n)-\psi(n))-\Phi(\varphi_1(n))\big)e(n\xi)\bigg|+N^{1-3/2\chi}\lesssim
 	\]
 	\[
 	\log(N)\sup_{1\le P\le N}\bigg|\sum_{0<|m|\le M}\frac{1}{2\pi im} \sum_{P<k\le P'\le 2P}\psi(k)^{-1}\Lambda_{a,q}(k)\big(e^{2\pi i(-m\varphi_1(k)+m\psi(k)+k\xi) }-e^{2\pi i (-m\varphi_1(k))+k\xi}\big)\bigg|+
 	\]
 	\[\log(N)N^{1-\chi-\varepsilon}
 	\]
 We will use summation by parts, let
 	\[
 	v_m(k)=\Lambda_{a,q}(k)e^{2\pi i(k\xi-m\varphi_1(k))}\text{, }V_m(x)=\sum_{P<k\le x} v_m(k)\text{, }d_m(t)=\psi^{-1}(t)(e^{2\pi im \psi(t)}-1)
 	\] We obtain
 	\[
 	\bigg| \sum_{P<k\le P'\le 2P}\psi(k)^{-1}\Lambda_{a,q}(k)\big(e^{2\pi i(-m\varphi_1(k)+m\psi(k)+k\xi) }-e^{2\pi i (-m\varphi_1(k))+k\xi}\big)\bigg|=\bigg|\sum_{P<k\le P'\le 2P}v_m(k)d_m(k) \bigg|=
 	\]
 	\[
 	\bigg|V_m(P')d_m(P')+\int_P^{P'}V_m(t)d_m'(t)dt\bigg|\le |V_m(P')d_m(P')|+\int_P^{P'}|V_m(t)d_m'(t)|dt
 	\]
 	We have 
 	\[
 	|d_m(t)|\le |\psi^{-1}(t)2\pi m \psi(t)|\lesssim |m|
 	\]
 	and
 	\[
 	|d_m'(t)|\le|\psi'(t)/\psi^2(t)(e^{2\pi im\psi(t)}-1)|+|\psi^{-1}(t)2\pi m\psi'(t)e^{2\pi im\psi(t)}|\lesssim \bigg|\frac{\psi'(t)}{\psi(t)}\bigg||m|\lesssim\bigg|\frac{\varphi_2''(t)}{\varphi_2'(t)}\bigg||m|\lesssim\frac{|m|}{t}
 	\]
 	We can now estimate
 	\[
 	|V_m(P')d_m(P')|+\int_P^{P'}|V_m(t)d_m'(t)|dt\lesssim |m||V_m(P')|+\int_P^{P'}\sup_{P<t\le2P}|V_m(t)m|/tdt\lesssim |m|\sup_{P<t\le2P}|V_m(t)|
 	\]
 	and thus	
 	\[
 	\bigg|\sum_{\substack{p\in\mathbb{P}\cap B\cap [N]\\p\equiv b\Mod m}}\psi(p)^{-1}\log(p)e(p\xi)-\sum_{\substack{p\in\mathbb{P}\cap [N]\\p\equiv b\Mod m}}\log(p)e(p\xi)\bigg|\lesssim
 	\]
 	\[
 	\log(N)\sup_{1\le P\le N}\sum_{0<|m|\le M}\frac{1}{2\pi |m|}|m|\sup_{P<t\le2P}|V_m(t)|+\log(N)N^{1-\chi-\varepsilon}\lesssim
 	\]
 	
 	\begin{equation}\label{finalred}
 		\log(N)\sup_{1\le P\le N}\sum_{0<|m|\le M}\sup_{P<P_1\le2P}|V_m(P_1)|+\log(N)N^{1-\chi-\varepsilon}
 	\end{equation}
 	This is the final reduction, making the estimate of Lemma~\ref{LLambda}, the only missing element in our proof. Notice that since we have chosen $M=P^{1+\chi+\varepsilon-99\gamma_2/100}$ where $\chi>0$, $0<\varepsilon<\chi/100$ and $16(1-\gamma_1)+17(1-\gamma_2)+31\chi\le1$, Lemma \ref{LLambda} is directly applicable, and we get
 	
 	\[
 	|V_m(P_1)|=\bigg|\sum_{P<k\le P_1\le 2P}\Lambda_{a,q}(k)e^{2\pi i(k\xi-m\varphi_1(k))}\bigg|\lesssim
 	\]
 	\[
 	|m|^{1/2}\log(P_1)P_1^{4/3}\big(\sigma_1(P_1)\varphi_1(P_1)\big)^{-1/2}+|m|^{1/6}\log^6(P_1)P_1^{13/12}(\sigma_1(P_1)\varphi_1(P_1))^{-1/6}
 	\]
 	and thus for any number $\varepsilon_3>0$, we have
 	\[
 	\sum_{0<|m|\le M}\sup_{P<P_1\le2P}|V_m(P_1)|\lesssim \sum_{0<|m|\le M} |m|^{1/2}\log(P_1)P_1^{4/3}\big(\sigma_1(P_1)\varphi_1(P_1)\big)^{-1/2}+
 	\]
 	\[
 	\sum_{0<|m|\le M} |m|^{1/6}\log^6(P_1)P_1^{13/12}(\sigma_1(P_1)\varphi_1(P_1))^{-1/6}\lesssim
 	\]
 	\[
 	M^{3/2}\log(P)P^{4/3}\big(\sigma_1(P)\varphi_1(P)\big)^{-1/2}+M^{7/6}\log^6(P)P^{13/12}(\sigma_1(P)\varphi_1(P))^{-1/6}\lesssim
 	\]
 	\[
 	P^{3/2+3\chi/2+3\varepsilon/2-297\gamma_2/200+4/3}\log(P)\varphi_1(P)^{-1/2}\sigma_1(P)^{-1/2}
 	\]
 	\[
 	+P^{7/6+7\chi/6+7\varepsilon/6-693\gamma_2/600+13/12}\log^6(P)\varphi_1(P)^{-1/6}\sigma_1(P)^{-1/6}\lesssim_{\varepsilon_3}
 	\]
 	\[
 	P^{3/2+3\chi/2+3\varepsilon/2-297\gamma_2/200+4/3-\gamma_1/2+\varepsilon_3}
 	\]
 	\[
 	+P^{7/6+7/6\chi+7\varepsilon/6-693\gamma_2/600+13/12-\gamma_1/6+\varepsilon_3}
 	\]
 	since $\varphi_1(x)\gtrsim_{\varepsilon_3} x^{\gamma_1-\varepsilon_3}$, $\sigma_1^{-1/2}(x)\lesssim_{\varepsilon_3} x^{\varepsilon_3/4}$ and $\log^6(x)\lesssim_{\varepsilon_3}x^{\varepsilon_3/12}$, and thus
 	\[
 	\varphi_1^{-1/2}(x)\log(x)\sigma_1^{-1/2}(x)\lesssim_{\varepsilon_3}x^{-\gamma_1/2+\varepsilon_3/2+\varepsilon_3/12+\varepsilon_3/4}\lesssim_{\varepsilon_3}x^{-\gamma_1/2+\varepsilon_3}
 	\]
 	\[\varphi_1^{-1/6}(x)\log^6(x)\sigma_1^{-1/6}(x)\lesssim_{\varepsilon_3} x^{-\gamma_1/6+\varepsilon_3/6+\varepsilon_3/12+\varepsilon_3/12}\lesssim_{\varepsilon_3}x^{-\gamma_1/6+\varepsilon_3}\]
 	It suffices to show that there exists a positive number $\varepsilon_3$ such that 
 	\[
 	3/2+3\chi/2+3\varepsilon/2-297\gamma_2/200+4/3-\gamma_1/2+\varepsilon_3< 1-\chi-\varepsilon
 	\]
 	and
 	\[
 	7/6+7\chi/6+7\varepsilon/6-693\gamma_2/600+13/12-\gamma_1/6+\varepsilon_3<1-\chi-\varepsilon
 	\]
 	in order to conclude that
 	\[
 	\sum_{0<|m|\le M}\sup_{P<P_1\le2P}|V_m(P_1)|\lesssim P^{1-\chi-\varepsilon}
 	\]
 	Remembering that $\varepsilon<\chi/100$, for the first inequality, it suffices to show that there exists an $\varepsilon_3>0$ such that
 	\[
 	11/6+5\chi/2+5\chi/200-297\gamma_2/200-\gamma_1/2+\varepsilon_3<0
 	\]
 	which is equivalent to
 	\[
 	891/91(1-\gamma_2)+300/91(1-\gamma_1)+1515\chi/91+600\varepsilon_3/91<1
 	\]
 	and notice that
 	\[
 	891/91(1-\gamma_2)+300/91(1-\gamma_1)+(1515/91)\chi<17(1-\gamma_2)+16(1-\gamma_1)+31\chi\le1
 	\]
Thus for sufficiently small $\varepsilon_3$ the first inequality is satisfied. For the second inequality, in a similar fashion, it suffices to find a $\varepsilon_3>0$ such that
 	\[
 	5/4+1313\chi/600-693\gamma_2/600-\gamma_1/6+\varepsilon_3<0
 	\]
 	or equivalently
 	\[
 	693/43(1-\gamma_2)+100/43(1-\gamma_1)+1313\chi/43+600\varepsilon_3/43<1
 	\]
 	and notice that
 	\[
 	693/43(1-\gamma_2)+100/43(1-\gamma_1)+(1313/43)\chi<17(1-\gamma_2)+16(1-\gamma_1)+31\chi\le1
 	\]
 Thus for sufficiently small $\varepsilon_3$ the second inequality is also satisfied. Therefore, we have that
 	\[
 	\sum_{0<|m|\le M}\sup_{P<P_1\le2P}|V_m(P_1)|\lesssim P^{1-\chi-\varepsilon}
 	\]
 	Thus, by \ref{finalred}, we get
 	\[
 	\bigg|\sum_{\substack{p\in\mathbb{P}\cap B\cap [N]\\p\equiv b\Mod m}}\psi(p)^{-1}\log(p)e(p\xi)-\sum_{\substack{p\in\mathbb{P}\cap [N]\\p\equiv b\Mod m}}\log(p)e(p\xi)\bigg|\lesssim
 	\]
 	\[
 	\log(N)\sup_{1\le P\le N}P^{1-\chi-\varepsilon}+\log(N)N^{1-\chi-\varepsilon}=\log(N)N^{1-\chi-\varepsilon}\lesssim N^{1-\chi-\varepsilon/2}
 	\]
Let's choose $\chi'$ to be $\varepsilon/2$, we have shown that
 	\[
 	\sum_{\substack{p\in\mathbb{P}\cap B\cap [N]\\p\equiv b\Mod m}}\psi(p)^{-1}\log(p)e(p\xi)-\sum_{\substack{p\in\mathbb{P}\cap [N]\\p\equiv b\Mod m}}\log(p)e(p\xi)=O(N^{1-\chi-\chi'})
 	\]
 	which is the desired result. The proof of Lemma~\ref{PR2} is complete.
 \end{proof}	

\end{document}